\documentclass[
BCOR2pt, 
captions=nooneline, 
bibliography=totoc, 
numbers=noenddot, 
parskip=half, 
headings=normal, 
abstracton 
]{scrartcl} 

\usepackage[USenglish]{babel} 
\usepackage{graphicx} 
\usepackage{tikz}
\usepackage{framed}
\usepackage{floatrow}
\usepackage{remreset} 
\usepackage[nouppercase]{scrpage2} 
\usepackage{amsmath} 
\usepackage{amssymb} 
\usepackage[thmmarks, 
amsmath, 
hyperref 
]{ntheorem} 
\usepackage{bm} 
\usepackage{bbm} 
\usepackage{enumitem} 
\usepackage[hang]{subfigure} 
\usepackage{wrapfig} 
\usepackage{tabularx} 
\usepackage{color}
\usepackage{dcolumn} 
\usepackage{booktabs} 
\usepackage{listings} 
\usepackage{psfrag} 
\usepackage[pdftex, 
setpagesize=false, 
pdfborder={0 0 0}, 
pdfpagemode=UseOutlines, 
]{hyperref} 
\makeatletter
\newcommand\mytoday{\number\year-\ifcase\month\or 01\or 02\or 03\or 04\or 05\or 06\or 07\or 08\or 09\or 10\or 11\or 12\fi-\ifcase\day\or 01\or 02\or 03\or 04\or 05\or 06\or 07\or 08\or 09\or 10\or 11\or 12\or 13\or 14\or 15\or 16\or 17\or 18\or 19\or 20\or 21\or 22\or 23\or 24\or 25\or 26\or 27\or 28\or 29\or 30\or 31\fi} 
\makeatother
\setkomafont{sectioning}{\normalcolor\bfseries} 
\clearscrheadfoot 
\automark[section]{subsection} 
\setkomafont{captionlabel}{\bfseries} 
\newcolumntype{d}[2]{D{.}{.}{#1.#2}} 
\newcommand*{\abstractnoindent}{} 
\let\abstractnoindent\abstract
\renewcommand*{\abstract}{\let\quotation\quote\let\endquotation\endquote
\abstractnoindent}
\makeatletter
\renewcommand{\p@enumii}[1]{\theenumi(#1)}
\makeatother


\theoremstyle{break} 
\theoremheaderfont{\bfseries}
\theorembodyfont{}
\theoremseparator{}
\newtheorem{definition}{Definition}[section] 
\newtheorem{lemma}[definition]{Lemma}

\newtheorem{remark}[definition]{Remark}
\newtheorem{example}[definition]{Example}

\theoremstyle{nonumberbreak} 
\theoremsymbol{$\Box$}
\newtheorem{proof}{Proof}

\providecommand{\keywords}[1]{\textbf{Keywords: } #1}
\providecommand{\subjectclass}[1]{\textbf{2010 AMS Subject Classification: } #1}

\newcommand*{\IP}{\mathbb{P}}
\newcommand*{\IE}{\mathbb{E}}
\newcommand*{\IR}{\mathbb{R}}

\newcommand*{\IN}{\mathbb{N}}
\newcommand*{\I}[1]{\textbf{1}_{#1}}

\newcommand*{\F}{\mathfrak{F}}

\usetikzlibrary{snakes}
\usetikzlibrary{arrows,shapes,positioning}
\usetikzlibrary{decorations.markings}
\tikzstyle arrowstyle=[scale=2]
\tikzstyle directed=[postaction={decorate,decoration={markings,
    mark=at position 1 with {\arrow[arrowstyle]{stealth}}}}]

\begin{document}

\renewcommand{\figurename}{Fig.}

\thispagestyle{plain}
	\begin{center}
		{\bfseries\Large Subordinators which are infinitely divisible w.r.t.\ time: Construction, properties, and simulation of max-stable sequences and infinitely divisible laws}
		\par\bigskip
		\vspace{1cm}
		
		{\Large Jan-Frederik Mai\footnote{Technische Universit\"{a}t M\"{u}nchen, Parkring 11, 85478 Garching--Hochbr\"{u}ck, \texttt{mai@tum.de}.} and Matthias Scherer\footnote{Technische Universit\"{a}t M\"{u}nchen, Parkring 11, 85478 Garching--Hochbr\"{u}ck, \texttt{scherer@tum.de}.}}\\
		\vspace{0.2cm}
		{Version of \today}\\
	\end{center}
\begin{abstract}
The concept of a L\'evy subordinator is generalized to a family of non-decreasing stochastic processes, which are parameterized in terms of two Bernstein functions. Whereas the independent increments property is only maintained in the L\'evy subordinator special case, the considered family is always strongly infinitely divisible with respect to time (IDT), meaning that a path can be represented in distribution as a finite sum with arbitrarily many summands of independent and identically distributed paths of another process. Besides  distributional properties of the process, we present two applications to the design of accurate and efficient simulation algorithms. First, each member of the considered family corresponds uniquely to an exchangeable max-stable sequence of random variables, and we demonstrate how the associated extreme-value copula can be simulated exactly and efficiently from its Pickands dependence measure. Second, we show how one obtains different series and integral representations for infinitely divisible  probability laws by varying the parameterizing pair of Bernstein functions, without changing the law of one-dimensional margins of the process. As a particular example, we present an exact simulation algorithm for compound Poisson distributions from the Bondesson class, for which the generalized inverse of the distribution function of the associated Stieltjes measure can be evaluated accurately.
\end{abstract}
\keywords{strong infinitely divisible w.r.t.\ time; subordinator; infinitely divisible law; Pickands dependence function; Bondesson class; Bernstein function.}\newline
\subjectclass{60G51; 60G70; 60G09.}
\newpage
\section{Introduction}
We recall that a \emph{L\'evy subordinator} $L=\{L_t\}_{t \geq 0}$ is a non-decreasing stochastic process on a probability space $(\Omega,\mathcal{F},\IP)$ with independent and stationary increments, whose paths are almost surely right-continuous and start at $L_0=0$, see \cite{bertoin99} for a textbook treatment. Intuitively, L\'evy subordinators are the continuous-time analog of discrete-time random walks with non-negative increments. The law of $L$, that is its finite-dimensional distributions, is fully determined by the law of any random variable $L_t$ with $t>0$, whose Laplace transform is given by 
\begin{gather}
\IE\Big[ e^{-x\,L_t}\Big] = e^{-t\,\Psi_L(x)},\quad \Psi_L(x) = \mu_L\,x+\int_{(0,\infty]}\big( 1-e^{-u\,x}\big)\,\nu_L(\mathrm{d}u),\quad x \geq 0,
\label{LevyKhinchin}
\end{gather}
where the so-called \emph{L\'evy measure} $\nu_L$ satisfies the condition $\int_{0}^{\infty}\min\{u,1\}\,\nu_L(\mathrm{d}u)<\infty$ and $\mu_L \geq 0$ is a drift constant. The function $\Psi_L$ is a so-called \emph{Bernstein function}, see \cite{schilling10} for a textbook treatment, and the number $\nu(\{\infty\})$ is called the \emph{killing rate} of $L$, because it corresponds to an exponential rate at which $L$ jumps to the absorbing graveyard state $\{\infty\}$, i.e.\ is ``killed.'' The so-called \textit{L\'evy--Khinchin formula} (\ref{LevyKhinchin}) establishes a one-to-one correspondence between Bernstein functions and L\'evy subordinators, so that $\Psi_L$  (or equivalently the pair $(\mu_L,\nu_L)$) provides a convenient analytical description of the law of $L$. 
\par
The purpose of the present article is to embed the concept of a L\'evy subordinator into a larger family of non-decreasing processes that can be parameterized in terms of a pair $(\Psi_L,\Psi_F)$ of two Bernstein functions. On the one hand, the enlarged family of processes still satisfies the concept of being \emph{strongly infinite divisible with respect to time}, as explained below, which renders it a natural generalization from an algebraic viewpoint. On the other hand, our generalization is inspired by two practical applications: Firstly, the processes can be used to construct and simulate multivariate extreme-value distributions. Second, they provide a reasonable framework to derive series representations for infinitely divisible laws on the positive half-axis, which can be used for simulation.
\par
For a pair $(F,L)$ of a distribution function $F$ of some non-negative random variable with finite, positive mean $\int_0^{\infty}\big(1-F(x)\big)\,\mathrm{d}x \in (0,\infty)$ and a L\'evy subordinator $L$ without drift, the present article studies distributional properties of the stochastic process
\begin{gather}
H_t:=H_t^{(F,L)}=\int_{0}^{\infty}-\log\Big\{ F\Big( \frac{s}{t}-\Big)\Big\}\,\mathrm{d}L_s,\quad t \geq 0,
\label{definition_H}
\end{gather}
the integral being defined pathwise in the usual Riemann--Stieltjes sense, and $F(x-):=\lim_{u \nearrow x}F(u)$. By definition, $H_0=0$ (using the notations $1/0:=\infty$ and $F(\infty -):=1$), the path $t \mapsto H_t$ is almost surely non-decreasing and right-continuous, and $H_t \in [0,\infty]$ for $t \geq 0$. We call a pair $(F,L)$ \emph{admissible}, whenever $H_t$ is not almost surely equal to the trivial process $H_t=\infty\,\cdot\,\I{\{t>0\}}$. Our interest in this semi-parametric family of stochastic processes is fueled both by theoretical and practical aspects. In the following Section~\ref{sec_prelim} we study distributional properties, whereas Sections~\ref{sec_Pickands} and \ref{sec_ID} give applications of the presented class of processes to the design of simulation algorithms. More precisely, the following list outlines the organization of the remaining article and the contributions made:
\begin{itemize}
\item \textbf{Section \ref{sec_prelim}:} It turns out that $H=\{H_t\}_{t \geq 0}$ is \emph{strongly infinitely divisible with respect to time (strong IDT)}\footnote{This property is called \emph{time-stable} in \cite{molchanov18}.}, meaning that for arbitrary $n \in \IN$ we have the distributional equality
\begin{gather*}
\{H_t\}_{t \geq 0} \stackrel{d}{=} \Big\{ \sum_{i=1}^{n}H^{(i)}_{\frac{t}{n}}\Big\}_{t \geq 0},
\end{gather*}
where $H^{(i)}$, $i=1,\ldots,n$, denote independent copies of $H$. In particular, from the viewpoint of the theory on infinite divisibility, the considered family of stochastic processes is a natural extension of the concept of a L\'evy subordinator. L\'evy subordinators arise in the special case when $F$ corresponds to a Bernoulli distribution with success probability $\exp(-1)$, see Example~\ref{example_LS}. Strong IDT processes have first been introduced in  \cite{mansuy05} and further examples have been studied in \cite{ouknine08,ouknine12}. These references give some examples of strong IDT processes with an emphasis on Gaussian processes. A LePage series representation for strong IDT processes without Gaussian component, in particular for non-decreasing strong IDT processes, is derived in \cite{molchanov18} and has been refined in the non-decreasing case by \cite{mai18}.
\par
We demonstrate how several distributional properties of $H$ can be inferred conveniently from the parameterizing pair $(F,L)$, thus pave the way to an analytical treatment of $H$ via its parameters. In particular, in addition to the defining integral representation we present an integration-by-parts formula and a canonical LePage series representation for $H$ in the spirit of \cite{molchanov18}. Section~\ref{sec_filt} studies the natural filtration $\{\mathcal{F}_t^{H}\}_{t \geq 0}$ of the process $H$. While L\'evy subordinators have independent increments, we demonstrate how the support of the probability measure $\mathrm{d}F$ controls the ability of $H$ to ``\textit{see into the future}.'' In particular, for $L$ a compound Poisson subordinator and support of $\mathrm{d}F$ bounded, the increment $H_{t+h}-H_t$ can be decomposed into a sum of one part that is measurable with respect to $\mathcal{F}_t^{H}$, and another part that is independent thereof. This peculiar property might be one explanation as to why strong IDT processes are, so far, not very well studied.
\item \textbf{Section \ref{sec_Pickands}:} Due to \cite[Theorem 5.3]{maischerer13}, a sequence $\{Y_k\}_{k \in \IN}$ of $(0,\infty)$-valued random samples drawn from the random distribution function $x \mapsto 1-\exp(-H_{x})$ is an \emph{exchangeable min-stable exponential sequence}, meaning that $\min\{Y_1/t_1,\ldots,Y_d/t_d\}$ has an exponential distribution with rate $\ell(t_1,\ldots,t_d,0,0,\ldots) \in [0,\infty)$, $d \in \IN$, and $t_1,\ldots,t_d \geq 0$ (not all zero). Equivalently, $\{1/Y_k\}_{k \geq 1}$ is an \emph{exchangeable max-stable sequence}. If $Y_1$ has unit mean, i.e.\ if $-\log\big(\IE[\exp\{-H_1\}]\big)=1$, the associated function $\ell : [0,\infty)^{\IN}_{00} \rightarrow [0,\infty)$ is called a \emph{stable tail dependence function}, defined on the set $[0,\infty)^{\IN}_{00}$ of non-negative sequences that are eventually zero. The stable tail dependence function uniquely characterizes the law of $\{Y_k\}_{k \in \IN}$ and, equivalently, the law of $H$. Since min- (resp.\ max-) stability is closely related to multivariate extreme-value theory, an understanding of the law of $H$ is thus tantamount with the understanding of an associated family of multivariate extreme-value copulas. In particular, a simulation algorithm for the random vector $(Y_1,\ldots,Y_d)$ is equivalent to one for the associated extreme-value copula. 
\par
Section~\ref{sec_Pickands} shows how the random vector $(Y_1,\ldots,Y_d)$ can be simulated exactly. To this end, we make use of a simulation algorithm from \cite{dombry16}, which requires to simulate from the so-called \emph{Pickands dependence measure} associated with $(Y_1,\ldots,Y_d)$, a finite measure on the $d$-dimensional unit simplex. In the present situation, we demonstrate how this simulation can be achieved efficiently and accurately.  
\item \textbf{Section~\ref{sec_ID}:} Fixing $t=1$, the random variable $H_1$ has an infinitely divisible law on $[0,\infty]$, which is invariant with respect to many changes in the parameterizing pair $(F,L)$. This fact can be used to derive different series representations for the same infinitely divisible law from Definition~(\ref{definition_H}), when either $L$ is of compound Poisson type or the support of $\mathrm{d}F$ is bounded. In spirit, this methodology is quite similar to seminal ideas in \cite{bondesson82}, who proposes alternative series representations for infinitely divisible laws on $\IR$. Section~\ref{sec_ID} demonstrates how the $(F,L)$-parameterization of $H_1$ provides a very convenient setting to derive a simulation algorithm for distributions from the so-called \emph{Bondesson class}, whenever the Stieltjes measure is given in a more convenient form than the L\'evy measure. In fact, if $L$ is chosen as a compound Poisson subordinator with unit exponential jumps, the definition of $H_1$ defines a bijection between the Bondesson class and distributions $F$ having finite mean and left-end point of support equal to zero.
\end{itemize}

Finally, Section~\ref{sec_conc} concludes and an appendix contains the technical proofs.

\section{Anatomy of the process $H$}\label{sec_prelim}
\subsection{Technical preliminaries}
Throughout, we denote by $L=\{L_t\}_{t \geq 0}$ a (possibly killed) L\'evy subordinator without drift and with L\'evy measure $\nu_L$ on $(0,\infty]$, i.e.\ with killing rate $\nu(\{\infty\})$. We assume that $\nu_L$ is non-zero, i.e.\ $L$ is not identically zero. Its associated Bernstein function is denoted by
\begin{gather*}
\Psi_L(x):=\int_{(0,\infty]}\big( 1-e^{-u\,x}\big)\,\nu_L(\mathrm{d}u),\quad x \geq 0,
\end{gather*}
implicitly using the short-hand notations $\exp(-\infty):=0$ and $0\,\cdot \infty := 0$ in order to enforce $\Psi(0)=0$.
\begin{remark}[Why driftless?]
A positive drift $\mu_L$ of the L\'{e}vy subordinator $L$ would imply a drift of the process $H$, namely
\begin{gather}
\mu_H=\mu_L\,\int_0^{\infty}-\log \{ F ( s) \} \,\mathrm{d}s.
\label{driftH}
\end{gather}
On the one hand, this is inconvenient, because it requires an additional integrability condition on $F$, so that (\ref{driftH}) exists at all. We will later postulate that $F$ satisfies $\int_0^{\infty}\big(1-F(s)\big)\,\mathrm{d}s<\infty$, which is a weaker condition. For instance, the distribution function $F(x)=\exp(-x^{-2})$ is admissible in the sequel but does not satisfy (\ref{driftH}). On the other hand, the assumption of a driftless L\'{e}vy subordinator is without loss of generality. To explain this, we will see that $H$ falls into the family of non-decreasing strong IDT processes. It follows from a structural result in \cite{mai18} that, just like for the subfamily of L\'{e}vy subordinators, such processes can be decomposed uniquely into $H_t = \mu_H\,t+\tilde{H}_t$, with a drift $\mu_H \geq 0$ and a non-decreasing strong IDT process $\tilde{H}$ without drift. This allows us to concentrate our study on the driftless case, because the more general case is simply obtained by adding a drift a posteriori.
\end{remark}

For later reference, we introduce the following sets of distribution functions (cdfs)
\begin{align*}
\F_c:&=\Big\{F \mbox{ cdf of some }X \in [0,\infty)\,:\,c=\IE[X]=\int_0^{\infty}\big(1-F(s)\big)\,\mathrm{d}s \Big\},\\
\F:&=\bigcup_{c \in (0,\infty)}\F_c.
\end{align*}
It is convenient to study the law of $H$ in terms of the function
\begin{gather*}
\ell(t_1,\ldots,t_d,0,0,\ldots):=-\log\Big(\IE\Big[ e^{-\sum_{k=1}^{d}H_{t_k}}\Big]\Big) 
\end{gather*}
for $d \in \IN$, $t_1,\ldots,t_d \geq 0$ arbitrary, and $F \in \F$.
\begin{lemma}[Laplace exponent of $H$]\label{STDF}
Let $d \in \IN$, $t_1,\ldots,t_d \geq 0$ arbitrary, and $F \in \F$. Then
\begin{align*}
\ell(t_1,\ldots,t_d,0,0,\ldots)= \int_{(0,\infty]}\int_0^{\infty}\Big(1-\prod_{k=1}^{d}F\big( \frac{s}{t_k}\big)^y\Big)\,\mathrm{d}s \,\nu_L(\mathrm{d}y).
\end{align*}
\end{lemma}
\begin{proof}
See the Appendix.
\end{proof}
We view $\ell$ as a mapping from the set $[0,\infty)^{\IN}_{00}$ of sequences, which are eventually zero, to $[0,\infty]$. The substitution $u=s/t$ shows that $\ell$ is homogeneous of order one, that is 
\begin{gather*}
\ell(t\cdot t_1,t\cdot t_2,\ldots,t\cdot t_d,0,0,\ldots) = t\,\ell(t_1,t_2,\ldots,t_d,0,0,\ldots)
\end{gather*}
for $t,t_1,\ldots,t_d \geq 0$. In the case $d=1$ and $t_1=x$ this implies in particular that the random variable $H_x$ is infinitely divisible for each $x \geq 0$. The mapping $\ell$ specifies the law of $H$ uniquely, i.e.\ its finite-dimensional distributions. This is due to the fact that the law of the random vector $(H_{t_1},\ldots,H_{t_d})$ is uniquely determined by the values of its multivariate Laplace transform on $\IN^{d}$, which follows from the Stone--Weierstrass Theorem (polynomials are dense in the space of continuous functions on $[0,1]^d$). Consequently, the mapping $\ell$, which in turn by Lemma \ref{STDF} is specified by $F$ and $\nu_L$, is a convenient analytical description for the law of $H$. 
\par
Since $F$ is right-continuous, so is $t \mapsto F(s/t-)$ for each fixed $s$. This implies that $t \mapsto H_t$ is almost surely right-continuous as integral over right-continuous functions. The condition $\IE[X]>0$ in the definition of $\F$ implies $F(0)<1$, which results in $\lim_{t \rightarrow \infty}H_t=\infty$ a.s.. The condition $\IE[X]<\infty$ in the definition of $\mathcal{F}$ is necessary (but needs not be sufficient) to have $\IP(H_t<\infty)>0$ for $t>0$. In order to explain this, we recall \cite[Lemma 3]{mai17}, which is required several times later on. 
\begin{lemma}[Bernstein functions associated with $\F$]\label{BF_mai18}
For $F \in \F$ the function
\begin{gather}\label{Psi_F_A}
\Psi_F(x):=\int_0^{\infty}\big(1-F(s)^{x}\big)\,\mathrm{d}s,\quad x \geq 0,
\end{gather}
is a Bernstein function, and the mapping $F \mapsto \Psi_F$ is a bijection between $\F$ and the set of Bernstein functions without drift. The L\'evy measure $\nu_F$ associated with $\Psi_F$ is determined in terms of $F$ by the equation $\nu_F((t,\infty])=F^{-1}(\exp(-t))$, $t > 0$, where $F^{-1}$ denotes the generalized inverse of $F$. The inverse mapping $\nu \mapsto F_{\nu}$ from the set of L\'evy measures on $(0,\infty]$ to $\F$ is given by
\begin{gather*}
F_{\nu}(t)=\begin{cases}
0 &\mbox{, if }t<\nu(\{\infty\})\\
e^{-S_{\nu}^{-1}(t)} &\mbox{, if }\nu(\{\infty\})\leq t < \nu((0,\infty])\\
1 & \mbox{, else}
\end{cases},
\end{gather*}
where $S_{\nu}^{-1}$ denotes the generalized inverse of $S_{\nu}(t):=\nu((t,\infty])$. 
\end{lemma}
\begin{proof}
See \cite[Lemma 3]{mai17}.
\end{proof}

Lemma \ref{BF_mai18} implies in particular that the integral $\int_0^{\infty}\big(1-F(s)^x\big)\,\mathrm{d}s$ is finite for all $x>0$ if and only if it is finite for a single fixed $x>0$. Now $\IP(H_t<\infty)>0$ is equivalent to $\ell(t,0,0,\ldots)<\infty$ and from Lemma~\ref{STDF} we see that this necessarily requires $\int_0^{\infty}\big(1-F(s)^y\big)\,\mathrm{d}s$ to be finite for those $y$ on which $\nu_L(\mathrm{d}y)$ puts mass, hence, in particular $\IE[X]=\int_0^{\infty}\big(1-F(s)\big)\,\mathrm{d}s<\infty$.
\par
While the condition $\IE[X]<\infty$ in the definition of $\F$ is necessary to prevent the non-interesting case $H_t=\infty\cdot\,\I{\{t>0\}}$, we have claimed that this this needs not be sufficient but depends on the specific choice of $L$ and $F$. To this end, we introduce the set
\begin{gather*}
\F_L:=\Big\{F \in \F\,:\,\ell(1,0,0,\ldots)=\int_{(0,\infty]}\Psi_F(y)\,\nu_L({d}y)<\infty \Big\}
\end{gather*}
of \emph{$L$-admissible} distribution functions. 

\begin{lemma}[Admissibility]\label{lemma_technicality}
The following statements are equivalent.
\begin{itemize}
\item[(a)] $(F,L)$ is admissible, i.e.\ $H_t \neq \infty\cdot \I{\{t>0\}}$.
\item[(b)] $F \in \F_L$, i.e.\ $\ell(1,0,0,\ldots)<\infty$.
\item[(c)] $\ell(t_1,\ldots,t_d,0,0,\ldots)<\infty$ for arbitrary $d \in \IN$ and $t_1,\ldots,t_d \geq 0$.
\end{itemize}
\end{lemma}
\begin{proof}
See the Appendix.
\end{proof}

\begin{example}[$\F_L$ in compound Poisson case]\label{ex_CPP}
Let $L$ be a (driftless) compound Poisson subordinator with intensity $\beta>0$ and jump size distribution $\IP(J \in \mathrm{d}y)$, where $J$ denotes a generic jump size random variable on $(0,\infty)$ with Laplace transform $\varphi_J$. It follows that $\Psi_L=\beta\,(1-\varphi_J)$ and
\begin{gather*}
\F_L = \Big\{ F\in \F\,:\,\IE\big[ \Psi_F(J)\big]<\infty\Big\}.
\end{gather*}
Since $\Psi_F(0)=0$, for $x \geq 1$ we obtain $\Psi_F(x) \leq x\,\Psi_F(1)$ by concavity of $\Psi_F$, which implies $\F_L=\F$ for those compound Poisson subordinators whose jump size distribution has finite mean.
\end{example}

As already mentioned, for each $t \geq 0$ the random variable $H_t$ is infinitely divisible. We denote the Bernstein function associated with $H_1$ by 
\begin{gather}
\Psi_H(x):=-\log\Big(\IE\Big[ e^{-x\,H_1}\Big]\Big)=\int_{(0,\infty]} \Psi_F(x\,y)\,\nu_L(\mathrm{d}y),\quad x \geq 0.
\label{anotherrepr_PsiH}
\end{gather}
The last equality follows for $x \in \IN$ from Lemma~\ref{STDF} in the special case $t_1=\ldots=t_d=1$ and for general $x \geq 0$ by the facts that (i) the right-hand side of (\ref{anotherrepr_PsiH}) defines a Bernstein function as mixture of Bernstein functions and (ii) a Bernstein function is uniquely determined by its values on $\IN$, see \cite[p.\ 36]{gnedin08}.
\par
Lemma~\ref{lemma_repr} derives two alternative stochastic representations of the process $\{H_t\}_{t \geq 0}$. The \textit{LePage representation} in part (b) is a particular special case of a result from \cite[Theorem 4.2]{molchanov18}.

\begin{lemma}[Alternative stochastic representations]\label{lemma_repr}
Let $(F,L)$ be an admissible pair, i.e.\ $F \in \F_L$.
\begin{itemize}
\item[(a)] \textbf{Integration-by-parts-formula:}
\begin{align*}
H_t=\int_0^\infty -\log\big(F(\frac{s}{t}-)\big)\,\mathrm{d}L_s=\int_0^\infty L_{s\,t} \,\mathrm{d}\big(\log(F(s))\big),\quad t>0.
\end{align*}
\item[(b)] \textbf{LePage series representation:}\\
Let $\{Z_k\}_{k \geq 1}$ be an iid sequence drawn from the probability measure $\frac{\Psi_F(z)}{\Psi_H(1)}\,\nu_L(\mathrm{d}z)$. Independently, let $\{\epsilon_k\}_{k \geq 1}$ be an iid sequence of unit exponential random variables. We then have the following equality in distribution:
\begin{gather*}
\{H_t\}_{t \geq 0} \stackrel{d}{=} \Big\{ -\sum_{k \geq 1} \log\Big[  F\Big( \frac{(\epsilon_1+\ldots+\epsilon_k)\,\Psi_F(Z_k)}{t\,\Psi_H(1)}-\Big)^{Z_k}\Big]\Big\}_{t \geq 0}.
\end{gather*}
\end{itemize}
\end{lemma}
\begin{proof}
See the Appendix.
\end{proof}

\begin{example}[The special case of a L\'evy subordinator]\label{example_LS}
Suppose that for $x\geq 0$, $F(x)=\exp(-1)+(1-\exp(-1))\,\I{[1,\infty)}(x)$ with associated Bernstein function
\begin{gather*}
\Psi_F(x)=\int_0^{\infty}\big(1-F(s)^x\big)\,\mathrm{d}s=1-e^{-x},\quad x\geq 0.
\end{gather*}
An arbitrary (driftless) L\'evy subordinator $L$ leads to an admissible pair $(F,L)$ and obviously we have $H=L$. So statement (b) in Lemma \ref{lemma_repr} provides an infinite series representation for an arbitrary L\'evy subordinator $L$, namely
\begin{gather}
\{L_t\}_{t \geq 0} = \Big\{\sum_{k \geq 1}Z_k\,\I{\big\{(\epsilon_1+\ldots+\epsilon_k)\,\frac{1-e^{-Z_k}}{\Psi_L(1)} \leq t\big\}} \Big\}_{t \geq 0}.
\label{levy_repr}
\end{gather}
In the special case when $L=N$ is a standard (unit intensity) Poisson process, this formula boils down to the well-known counting process representation
\begin{gather}\label{standardPoisson}
\{N_t\}_{t \geq 0} = \Big\{\sum_{k \geq 1}\I{\{\epsilon_1+\ldots+\epsilon_k\leq t\}} \Big\}_{t \geq 0}.
\end{gather}
Representation~(\ref{levy_repr}) is a quite natural generalization of (\ref{standardPoisson}), and by Lemma~\ref{lemma_repr}(b) it is general enough to comprise all L\'evy subordinators. In particular, it is worth mentioning that the probability law of the $Z_k$ is $(1-\exp(-z))\,\nu_L(\mathrm{d}z)/\Psi_L(1)$, which can be any probability law on $(0,\infty]$. This means that, conversely, if $\rho$ is an arbitrary probability law on $(0,\infty]$ and $\{Z_k\}_{k \geq 1}$ is an iid sequence drawn from $\rho$, (\ref{levy_repr}) defines a L\'evy subordinator $L$ without drift and with associated L\'evy measure $\nu_L(\mathrm{d}z)=(1-\exp(-z))^{-1}\,\rho(\mathrm{d}z)$. 
\end{example}

\subsection{Distributional properties}

Recall that an infinitely divisible distribution on $[0,\infty]$ is of compound Poisson type if its associated Bernstein function $\Psi$ is bounded, i.e.\ $\lim_{x \rightarrow \infty}\Psi(x)<\infty$. 

\begin{lemma}[When does $H_t$ have a compound Poisson distribution?] \label{lemma_CP}
Let $(F,L)$ be admissible. The Bernstein function $\Psi_H$ is bounded if and only if the following two conditions are satisfied:
\begin{itemize}
\item[(a)] $\Psi_{L}$ is bounded, i.e.\ $L$ is a compound Poisson subordinator.
\item[(b)] $\Psi_F$ is bounded, i.e.\ the random variable $X \sim F$ has bounded support.
\end{itemize}
\end{lemma}
\begin{proof}
See the Appendix.
\end{proof}

Recall that an infinitely divisible law on $[0,\infty]$ is said to have killing, if it assigns positive mass to $\{\infty\}$, which is the case if and only if its associated L\'evy measure $\nu$ satisfies $\nu(\{\infty\})>0$. In terms of the associated Bernstein function $\Psi$, this means $\Psi(x)>\epsilon>0$ for all positive $x>0$. In this case, we also say that the Bernstein function $\Psi$ has killing.

\begin{lemma}[When does $H_t$ have a positive killing rate?] \label{lemma_killing}
Let $(F,L)$ be admissible. The Bernstein function $\Psi_H$ has killing if and only if at least one of the following two conditions is satisfied:
\begin{itemize}
\item[(a)] $\Psi_F$ has killing, i.e.\ the left end point of the support of $X \sim F$ is strictly positive.
\item[(b)] $\Psi_L$ has killing, i.e.\ $\nu_L(\{\infty\})>0$.
\end{itemize}
\end{lemma}
\begin{proof}
See the Appendix.
\end{proof}

Since we are only interested in a description of the probability law of $H=H^{(F,L)}$, it is helpful to briefly ponder on potential redundancies, i.e.\ to investigate the question: \textit{When do two different pairs $(F,L)\neq (\tilde{F},\tilde{L})$ lead to exactly the same probability law of the associated processes $H$?} To address this issue in a mathematically rigorous manner, we introduce the equivalence relation
\begin{gather*}
(F,L) \sim (\tilde{F},\tilde{L}) :\Leftrightarrow \mbox{ the law of }H^{(F,L)}\mbox{ equals that of }H^{(\tilde{F},\tilde{L})}.
\end{gather*}
The equivalence class of an admissible pair $(F,L)$ is denoted by $[F,L]$, or also by $[F,\Psi_L]$, in the sequel. Determining the equivalence class $[F,L]$ of a pair $(F,L)$ is surprisingly involved. It is not difficult, however, to see that for an admissible pair $(F,L)$ the equivalence class $[F(c_1\,.)^{c_2},c_1\,\Psi_L(./c_2)]$ is invariant with respect to $c_1,c_2>0$. Depending on the admissible pair, there can be even more redundancies, as the following example demonstrates.

\begin{example}[The curious case $F=$ Fr\'echet distribution]\label{ex_Gumbel}
With a parameter $\theta \in (0,1)$ and $c_{\theta}:=\Gamma(1-\theta)^{-1/\theta}$ consider the Fr\'echet distribution $F(x):=\exp(-c_{\theta}\,x^{-1/\theta})\I{\{x>0\}}$ and observe that $\Psi_{F}(x)=x^{\theta}$. For arbitrary $y>0$ it is not difficult to compute
\begin{gather*}
\int_0^{\infty}\Big(1-\prod_{k=1}^{d}F\Big(\frac{s}{t_k} \Big)^{y} \Big)\,\mathrm{d}s = y^{\theta}\,\Big(\sum_{k=1}^{d}t_k^{\frac{1}{\theta}} \Big)^{\theta}.
\end{gather*}
Let $L$ be a L\'evy subordinator without drift and such that $F \in \F_L$, i.e.\ such that $(F,L)$ is admissible. It follows that
\begin{gather*}
\ell(t_1,\ldots,t_d,0,0,\ldots)=\int_{(0,\infty]}y^{\theta}\,\nu_L(\mathrm{d}y)\,\Big(\sum_{k=1}^{d}t_k^{\frac{1}{\theta}} \Big)^{\theta}=\Psi_H(1)\,\Big(\sum_{k=1}^{d}t_k^{\frac{1}{\theta}} \Big)^{\theta}.
\end{gather*}
Consequently, the function $\ell$, hence the law of $H$, depends on the choice of $L$ only via the scalar $\Psi_H(1)$. In particular, in order to study the probability law of $H$ it is sufficient to choose one particular $L$. Choosing $L=N$, i.e.\ a standard Poisson process, we know\footnote{See, e.g., \cite{mai17}.} that $\{H_t\}_{t \geq 0} \stackrel{d}{=}\{M_{\theta}\,t^{1/\theta}\}_{t \geq 0}$ with a $\theta$-stable random variable $M_{\theta}$. In contrast to L\'evy subordinators, which have independent increments, this stochastic process looks peculiar at first glimpse. The whole path of the process is already known if one just observes $H_t$ for one $t>0$. This phenomenon is studied in more detail in Section~\ref{sec_filt}.
\end{example}

\subsection{The natural filtration of $H$} \label{sec_filt}
In this paragraph we investigate the amount of information one can obtain by observing the process $H$ up to some time $t>0$. The result is remarkable and different to most classes of stochastic processes commonly used. We begin with an auxiliary result that shows how much information about $L$ we can filter out of an observation of $H$.

\begin{lemma}[Filtering out $L$ from $H$]\label{lemma_filt}
Let $L$ be a (driftless) compound Poison subordinator with positive jump sizes. We define by $\mathcal{F}_t^H=\sigma(H_s:0\leq s\leq t)$ the information from observing $H$ up to time $t>0$, similarly we define and interpret $\mathcal{F}^L_t$. Moreover, we assume $u_F<\infty$, where $u_F$ denotes the right-end point of the support of $\mathrm{d}F$. Then
\begin{align}\label{Filtration}
\mathcal{F}^H_t=\mathcal{F}^L_{u_F\,t},\quad t\geq 0.
\end{align}
\end{lemma}
\begin{proof}
See the Appendix.
\end{proof}
Intuitively, Lemma~\ref{lemma_filt} means that observing $H$ up to $t>0$ allows us to anticipate the process $L$ up to time $t\,u_F$. In the sequel, we decompose the increments of $H$ into two parts. To this end, we assume $u_F<\infty$ and pick the L\'evy subordinator $L$ arbitrary. We then observe for $x>0$ using Lemma~\ref{lemma_repr}(a) that
\begin{align*}
&H_{t+x}-H_t=\int_0^{u_F}\big(L_{s\,(t+x)} -L_{{u_F}\,t}+L_{{u_F}\,t}-L_{s\,t}\big) \,\mathrm{d}\big(\log F(s)\big)\\
	&=\int_{\frac{{u_F}\,t}{t+x}}^{u_F} \big(L_{s\,(t+x)} -L_{{u_F}\,t}\big) \,\mathrm{d}\big(\log F(s)\big)+\int_0^\frac{{u_F}\,t}{t+x} \big(L_{s\,(t+x)} -L_{{u_F}\,t}\big) \,\mathrm{d}\big(\log F(s)\big)\\
	&\qquad +\int_0^{u_F} \big(L_{{u_F}\,t}-L_{s\,t}\big) \,\mathrm{d}\big(\log F(s)\big)\\
	&= \underbrace{\int_{{u_F}\,t}^{{u_F}\,(t+x)}\big(L_{v}-L_{{u_F}\,t}\big)\,\mathrm{d}\big(\log F(\frac{v}{t+x})\big)}_{\text{positive random variable, independent of }\mathcal{F}_{{u_F}\,t}^{L}}+\underbrace{\int_{0}^{{u_F}\,t}\big(L_{{u_F}\,t}-L_v\big)\,\mathrm{d}\big(\log \frac{F(v/t)}{F(v/(t+x))}\big)}_{\text{positive random variable,  }\mathcal{F}_{{u_F}\,t}^{L}\text{-measurable}}\\
	&=: X_1+X_2. 
\end{align*}
\begin{remark}[Situation of Lemma~\ref{lemma_filt}]
In the situation of Lemma~\ref{lemma_filt}, i.e.\ when $L$ is a (driftless) compound Poisson process, we have $\mathcal{F}_{{u_F}\,t}^{L}=\mathcal{F}_{t}^{H}$. This means that in such a situation, the increments $H_{t+x}-H_t$ of the process $H$ can be split up into a part $X_2$ that can be anticipated from observing the past and a part $X_1$ that is independent of the past. This is quite an astonishing property for a stochastic process and far off the ``usual'' property of L\'evy processes having independent increments. 
\end{remark}
We can continue to investigate $X_1$ and $X_2$ and find:
\begin{align*}
X_1&\stackrel{d}{=} \int_0^{u_F} L_{x\, y}\,\mathrm{d}\big(\log F(\frac{x\,y+{u_F}\,t}{t+x})\big)=-\int_0^{u_F} \log F(\frac{x\,y+{u_F}\,t}{x+t}-)\, \mathrm{d} L_{x\,y}\geq 0,\\
X_2&=\int_0^{{u_F}\,t}-\log F(\frac{v}{t+x}-)\,\mathrm{d} L_{v}-H_t\geq 0.
\end{align*}
The second expression for $X_1$ particularly shows that it is infinitely divisible with associated Bernstein function 
\begin{gather}\label{bernstein_psi1}
\Psi_1(\alpha)=-\log \IE\big[e^{-\alpha \, X_1}\big]=x\,\int_0^{u_F} \Psi_L(-\alpha\log F(\frac{x\,y+{u_F}\,t}{x+t}))\,\mathrm{d}y.
\end{gather}

\begin{example}[The case $-\log F(x)=(1-x)_{+}$]\label{Example_German1}
This example was first considered in \cite{bernhart15}. Investing it in the present context (note that $u_F=1$), we find
\begin{align*}
\Psi_1(\alpha)&= x\,\int_0^1 \Psi_L\big(\alpha(1-\frac{x\,y+t}{x+t})\big)\,\mathrm{d} y =x\,\int_0^1 \Psi_L\big(\alpha(1-y)\,\frac{x}{x+t}\big)\,\mathrm{d} y,\\
X_2&=\frac{x}{x+t}\,\big(L_t-H_t\big).
\end{align*}
The function $\Psi_1$ tends to zero as $t$ increases to infinity. This is intuitive in the sense that the larger $t$, the more we know about the increment $H_{t+x}-H_{t}$ for fixed $x>0$. 
\end{example}

Equation (\ref{Filtration}) suggests that for $F$ with unbounded support, i.e.\ $F(t)<1$ for all $t\geq 0$, the whole path of the L\'evy subordinator, and, hence $H$, is known already when observing the path of $H$ on $[0,t]$, for $t>0$. Indeed, the following two examples confirm this presumption that is later shown with Lemma~\ref{Gumbel_generalized}.

\begin{example}[The case $L=N$ and $F(x)=1-\exp(-x)$]\label{ex_galambos}
Let $L=N$ be a Poisson process with unit intensity, whose jump time sequence we denote by $\{\tau_k\}_{k \geq 1}$, and $F$ the distribution function of the unit exponential law. In particular, $\mathrm{d}F$ has unbounded support. One can show that $\mathcal{F}_t^{H}=\mathcal{F}_{\infty}^{H}$ for all $t>0$. In words, this means that the whole path of $H$ is determined completely by the path on $[0,t]$ for arbitrarily small $t>0$. To accomplish this, we show in the Appendix that the function
\begin{gather*}
z \mapsto H_z=\sum_{k \geq 1}-\log\Big( 1-e^{-\frac{\tau_k}{z}}\Big)
\end{gather*}
is holomorphic on $\mathbb{C}_+:=\{z \in \mathbb{C}\,:\,\mathcal{R}(z)>0\}$. Since holomorphic functions on $\mathbb{C}_+$ are determined everywhere, once they are determined on a small real interval, such as $(0,t) \subset \mathbb{C}_+$ for $t>0$, the claim follows. 
\end{example}

\begin{example}[The case $F(x)=\exp(-c_\theta\, x^{-1/\theta})$]\label{Example_Gumbel}
Let $M_\theta$ be a random variable with Laplace transform $x \mapsto \exp(- x^{\theta} )$, $\theta\in(0,1)$, then the associated IDT subordinator is given as $H_t=M_\theta\,t^{1/\theta}$, see Example~\ref{ex_Gumbel}. This is a peculiar stochastic process, as observing it at some $t>0$ corresponds to knowing it everywhere. We have furthermore seen in Example~\ref{ex_Gumbel} that the choice of L\'evy subordinator $L$ is arbitrary, provided admissibility. In particular, we are free to choose a compound Poisson process with unit intensity and unit exponentially distributed jumps, i.e.\ $\Psi_L(x)=x/(x+1)$, which is a convenient choice for the following considerations. We truncate $F$ via $F_n(x):=\I{\{x\geq n\}}+\I{\{x< n\}}\,F(x)$. Clearly, we then have $u_{F_n}=n<\infty$ and $F_n$ converges to $F$ pointwise. We find for $X_2$ and the Bernstein function of $X_1$, see (\ref{bernstein_psi1}),
\begin{align*}
\Psi^{(n)}_1(\alpha)&= x\,\int_0^n\frac{c_\theta\,(x+t)^{1/\theta}}{c_\theta \,(x+t)^{1/\theta}+(x\,y+n\,t)^{1/\theta}}\,\mathrm{d} y \longrightarrow 0, \qquad (n\rightarrow\infty),\\
X_2&=c_\theta\big((t+x)^{1/\theta}-t^{1/\theta}\big)\,\int_0^{n\,t}v^{-1/\theta}\,\mathrm{d}\, L_v \longrightarrow H_{t+x}-H_t, \qquad (n\rightarrow\infty),
\end{align*}
both observations confirming our knowledge about $H_t$.
\end{example}
\begin{lemma}[The case of unbounded support of $\mathrm{d}F$]\label{Gumbel_generalized}
Like in Lemma \ref{lemma_filt}, we assume that $L$ is a compound Poisson subordinator, but now let $\mathrm{d}F$ have unbounded support, i.e.\ $u_F=\infty$. Then for arbitrary $t>0$ we have
\begin{align*}
\mathcal{F}^H_t=\mathcal{F}^L_{\infty}=\mathcal{F}^H_{\infty}.
\end{align*}
\end{lemma}
\begin{proof}
See the Appendix.
\end{proof}

\section{Application 1: Simulation of extreme-value copulas}\label{sec_Pickands}
Throughout this section, we fix one admissible pair $(F,L)$. The fact that $\ell$ in Lemma~\ref{STDF} is homogeneous of order $1$ implies by virtue of \cite[Theorem 5.3]{maischerer13} that the infinite exchangeable sequence of random variables
\begin{gather*}
Y_k:=\inf\{t>0\,:\,H_t>\epsilon_k\},\quad k \in \IN,
\end{gather*}
with $\{\epsilon_k\}_{k \in \IN}$ independent unit exponentials, independent of $H$, is \emph{min-stable multivariate exponential} with survival function given by
\begin{gather}
\IP(Y_1>t_1,\ldots,Y_d>t_d) = e^{-\ell(t_1,\ldots,t_d,0,0,\ldots)},\quad t_1,\ldots,t_d \geq 0,
\label{definitionY}
\end{gather}
which means that $\min\{Y_1/t_1,\ldots,Y_d/t_d\}$ has a univariate exponential distribution with rate $\ell(t_1,\ldots,t_d,0,0,\ldots)$, not all $t_k$ equal to zero. The exponential rate of each component $Y_k$ equals $\Psi_H(1)$ and it is convenient to normalize it to $\Psi_H(1)=1$, which then implies that the $d$-variate function $\ell_d(t_1,\ldots,t_d):= \ell(t_1,\ldots,t_d,0,0,\ldots)$ defines a so-called \emph{stable tail dependence function} in dimension $d$. If for $(F,L)$ the law of $H=H^{(F,L)}$ does not satisfy $\Psi_H(1)=1$, we can always change from $F$ to $F(c\,.)$ for some $c>0$ such that $H^{(F(c\,.),L)}$ has this property. This $c$ is even unique, i.e.\ for any given $(F,L)$ there is a unique $c>0$ such that $H^{(F(c\,.),L)}$ satisfies $\Psi_H(1)=1$. Given this normalization, the function $C_d(u_1,\ldots,u_d):=\exp\big\{-\ell_d\big(-\log(u_1),\ldots,-\log(u_d)\big)\big\}$, $u_1,\ldots,u_d \in (0,1)$, defines a so-called \emph{extreme-value copula}. For background on the latter, the reader is referred to \cite{Joe97,nelsen06,gudendorf09}. Loosely speaking, extreme-value copulas are the dependence structures behind the limit of appropriately normalized componentwise maxima of independent and identically distributed random vectors, which is of paramount interest in multivariate extreme-value theory. The relationship between strong IDT subordinators and multivariate extreme-value theory has already been investigated in the present authors' references \cite{maischerer13,bernhart15,mai17,mai18}. On one hand, the probability space (\ref{definitionY}) can directly be used to simulate the random vector $\mathbf{U}:=\exp(-\mathbf{Y}) \sim C_d$, where $\mathbf{Y}=(Y_1,\ldots,Y_d)$. However, due to the infinite series representation of $H$ in Lemma~\ref{lemma_repr}(b) and the fact that the increments of $H$ are typically not independent, this is a non-trivial task in general, although feasible in particular cases, an example with $L$ a Poisson process and support of $\mathrm{d}F$ bounded is provided in \cite[Section 3.1]{mai17}. However, there is an alternative approach to accomplish the simulation in the general case, as described in the sequel.
\par
It is well-known from \cite{dehaan77,ressel13} that the stable tail-dependence function $\ell_d$ is uniquely associated with a probability measure on the unit simplex $S_d:=\{(q_1,\ldots,q_d) \in [0,1]^d\,:\,q_1+\ldots+q_d=1\}$ subject to the constraint that each component has expected value $1/d$. To wit, there exists a random vector $\mathbf{Q}=(Q_1,\ldots,Q_d)$, uniquely determined in law, taking values in $S_d$ and satisfying $\IE[Q_k]=1/d$ for each $k=1,\ldots,d$, such that
\begin{gather*}
\ell_d(t_1,\ldots,t_d) = d\,\IE[\max\{t_1\,Q_1,\ldots,t_d\,Q_d\}],\quad t_1,\ldots,t_d \geq 0,
\end{gather*} 
which is called the \emph{Pickands representation} of $\ell_d$, named after \cite{pickands81}. It is important to notice that $\mathbf{Q}=\mathbf{Q}^{(d)}$ depends on the dimension $d$. In particular, in our situation where $d$ is arbitrary the first $d$ components of $\mathbf{Q}^{(d+1)}$ are not equal in distribution to $\mathbf{Q}^{(d)}$, not even when re-scaled. In order to simplify notation, however, we omit to highlight this dependence on $d$ for the rest of this paragraph.
\par
The simulation algorithm in \cite[Algorithm 1]{dombry16}, based on a seminal idea by \cite{schlather02}, shows how to simulate a random vector $\mathbf{U}\sim C_d$ exactly and efficiently, if one has at hand a simulation algorithm for the vector $\mathbf{Q}$. More precisely, it is shown that
\begin{align}
\mathbf{U} &\stackrel{d}{=} \Big( \exp\Big\{-\frac{1}{Z_1^{(M)}}\Big\},\ldots,\exp\Big\{-\frac{1}{Z^{(M)}_d}\Big\}\Big), \label{schlather}\\
\big( Z_1^{(n)},\ldots,Z^{(n)}_d\big) &\stackrel{d}{=} \Big( \max_{k=1,\ldots,n}\Big\{ \frac{d\,Q_1^{(k)}}{\epsilon_1+\ldots+\epsilon_k}\Big\},\ldots,\max_{k=1,\ldots,n}\Big\{ \frac{d\,Q_d^{(k)}}{\epsilon_1+\ldots+\epsilon_k}\Big\}\Big),\quad n\geq 1,\nonumber
\end{align}
where $\{\mathbf{Q}^{(k)}\}_{k \geq 1}$ denote independent copies of $\mathbf{Q}$, independently of $\{\epsilon_k\}_{k \geq 1}$ iid unit exponentials, and $M$ equals the smallest $n \in \IN$ for which $d/(\epsilon_1+\ldots+\epsilon_{n+1})$ is smaller than the minimal component of $\big( Z_1^{(n)},\ldots,Z^{(n)}_d\big)$. Thus, deriving an efficient and exact simulation algorithm for $\mathbf{Q}$ is essentially the key to deriving an efficient and exact simulation algorithm for the extreme-value copula $C_d$ associated with $H$ in dimension $d \geq 2$. The purpose of the present section is to demonstrate how this is possible. Concluding, one concrete application of the stochastic processes considered in the present article is to enlarge the repertoire of extreme-value copulas for which exact and efficient simulation strategies are available.

\subsection{The Pickands representation of $\ell_d$}
We assume that $\Psi_H(1)=1$ so that $\ell_d$ is a proper stable tail dependence function. This condition implies that
\begin{gather*}
\int_{(0,\infty]}\Psi_F(z)\,\nu_L(\mathrm{d}z) = 1,
\end{gather*}
so that $\Psi_F(z)\,\nu_L(\mathrm{d}z)$ is a probability measure on $(0,\infty]$. Notice that this probability measure has already been important in Lemma \ref{lemma_repr}(b). In the sequel, we show that it also occupies a commanding role when determining the Pickands representation of $\ell_d$. Denoting by $Z$ a generic random variable drawn from this probability law, observe that
\begin{align*}
\ell_d(t_1,\ldots,t_d) &=\int_{(0,\infty]}\iint_{[0,\infty)^d}\max\{t_1\,x_1,\ldots,t_d\,x_d\}\,\prod_{i=1}^{d}\mathrm{d}F^{y}(x_i) \,\nu_L(\mathrm{d}y)\\
& = \IE\Bigg[ \frac{1}{\Psi_F(Z)}\,\iint_{[0,\infty)^d}\max\{t_1\,x_1,\ldots,t_d\,x_d\}\,\prod_{i=1}^{d}\mathrm{d}F^{Z}(x_i) \Bigg]\\
& = d\,\IE\Bigg[ \iint_{[0,\infty)^d}\max_{i=1,\ldots,d}\Big\{t_i\,\frac{x_i}{\sum_{k=1}^{d}x_k}\Big\}\,\frac{1}{d}\,\sum_{k=1}^{d}\frac{x_k\,\mathrm{d}F^{Z}(x_k)}{\Psi_F(Z)}\,\prod_{i\neq k}^{d}\mathrm{d}F^{Z}(x_k)\Bigg].
\end{align*}
The important observation from this computation is that the vector
\begin{gather*}
(q_1,\ldots,q_d):=\Big(\frac{x_1}{\sum_{k=1}^{d}x_k},\ldots,\frac{x_d}{\sum_{k=1}^{d}x_k}\Big)
\end{gather*}
takes values in $S_d$ and, conditioned on $Z$, the measure $\frac{1}{d}\,\sum_{k=1}^{d}\frac{x_k\,\mathrm{d}F^{Z}(x_k)}{\Psi_F(Z)}\,\prod_{i\neq k}^{d}\mathrm{d}F^{Z}(x_k)$ is a probability measure on $[0,\infty)^d$. To see this, notice that conditioned on $Z$, the measure $\frac{x\,\mathrm{d}F^{Z}(x)}{\Psi_F(Z)}$ is a probability measure on $[0,\infty)$, since
\begin{gather*}
\int_{[0,\infty)}\frac{x\,\mathrm{d}F^{Z}(x)}{\Psi_F(Z)} = \frac{\int_0^{\infty}\big(1-F(s)^{Z}\big)\,\mathrm{d}s}{\Psi_F(Z)} = 1.
\end{gather*}
Consequently, we have found the unique Pickands dependence measure, as summarized in the following lemma.
\begin{lemma}[Pickands representation of $\ell_d$]\label{lemma_Pickands}
A random sample from ${\textbf Q}$ can be drawn according to the following algorithm.
\begin{enumerate}[label=(\roman*)]
\item Draw $D$ uniformly distributed on $\{1,\ldots,d\}$.
\item Draw a sample of the random variable $Z\sim \Psi_F(Z)\,\nu_L(\mathrm{d}z)$. 
\item Draw $d$ independent and identically distributed random variables $X_1,\ldots,X_d$ from the distribution function $F^{Z}$. 
\item Draw a random variable $M$ from the probability law $x\,\mathrm{d}F^{Z}(x)/\Psi_F(z)$.
\item Compute the random vector $(W_1,\ldots,W_d)$, defined by
\begin{gather*}
W_k:=\begin{cases}
X_k &\mbox{, if }k=D\\
M & \mbox{, else}\\
\end{cases}.
\end{gather*}
\item Return
\begin{gather*}
{\textbf Q}:=\Big( \frac{W_1}{\sum_{k=1}^{d}W_k},\ldots,\frac{W_d}{\sum_{k=1}^{d}W_k}\Big).
\end{gather*}
\end{enumerate} 
\end{lemma}

\begin{remark}[Expected runtime of the algorithm in Lemma~\ref{lemma_Pickands}]
When using the stochastic representation (\ref{schlather}) together with Lemma~\ref{lemma_Pickands} to simulate the extreme-value copula $C_d$, the runtime of the simulation algorithm is random itself. However, \cite[Proposition~4]{dombry16} shows that the expected value of $M$ in (\ref{schlather})  equals $d\,\IE[\max\{Y_1,\ldots,Y_d\}]$, which in the present situation can be computed in closed form by
\begin{align*}
\IE[M]&=d\,\IE[\max\{Y_1,\ldots,Y_d\}] = d\,\IE\Big[\int_0^{\infty}\Big(1-\big( 1-e^{-H_t}\big)^d\Big)\,\mathrm{d}t\Big]\\
& = -d\,\sum_{k=1}^{d}\binom{d}{k}\,(-1)^k\,\int_0^{\infty}\IE\Big[ e^{-k\,H_t}\Big]\,\mathrm{d}t =-d\,\sum_{k=1}^{d}\binom{d}{k}\,(-1)^k\,\int_0^{\infty}e^{-t\,\Psi_H(k)}\,\mathrm{d}t\\
& = -d\,\sum_{k=1}^{d}\binom{d}{k}\,\frac{(-1)^k}{\Psi_H(k)} \leq d^2.
\end{align*}
The last estimate follows from the estimate
\begin{gather*}
\IE\Big[\int_0^{\infty}\Big(1-\big( 1-e^{-H_t}\big)^d\Big)\,\mathrm{d}t\Big] = \IE\big[\Psi_{1-e^{-H_.}}(d)\big] \leq \underbrace{\IE[\Psi_{1-e^{-H}}(1)]}_{=1\text{ by normalization}}d,
\end{gather*}
where the inequality follows from the argument of the proof of Lemma~\ref{lemma_technicality}. Since the simulation of $\mathbf{Q}$ in Lemma~\ref{lemma_Pickands} itself is apparently of linear order in the dimension $d$, the total expected runtime for the exact simulation of the extreme-value copula $C_d$ according to the representation (\ref{schlather}) with the help of Lemma~\ref{lemma_Pickands} has expected order between $d^2$ and $d^3$ and can be computed explicitly in terms of the Bernstein function $\Psi_H$. 
\end{remark}

In the sequel, we work out some concrete examples, demonstrating the versatility of Lemma~\ref{lemma_Pickands}.
\subsection{Examples}
\begin{example}[The L\'evy subordinator case]\label{Example_Levy_case}
Consider the distribution function $F(x)=\exp(-1)+(1-\exp(-1))\,\I{[1,\infty)}(x)$, for $x\geq 0$, with associated Bernstein function
\begin{gather*}
\Psi_F(x)=1-e^{-x},\quad x\geq 0.
\end{gather*}
An arbitrary L\'evy subordinator $L$ leads to an admissible pair $(F,L)$ and obviously $H=L$. Furthermore, $\Psi_H(1)=1$ is satisfied whenever $\Psi_L(1)=1$. It is well-known that $C_d$ equals the survival copula of a so-called \emph{Marshall--Olkin distribution}, named after \cite{marshall67}, see \cite{maischerer09,mai11}. Furthermore, it is observed for $z>0$ that a random variable $X \sim F^z$ is Bernoulli-distributed with success probability $1-\exp(-z)$. In order to simulate from the Pickands measure, additionally required is also a simulation algorithm for $M \sim x\,\mathrm{d}F^{z}(x)/\Psi_F(z)$ with given $z$. With $X$ a Bernoulli random variable with success probability $1-\exp(-z)$, it is observed that
\begin{gather*}
\IP(M \leq x) = \int_{0}^{x}\frac{s}{\Psi_F(z)}\,\mathrm{d}F^{z}(s) = \frac{1}{1-e^{-z}}\,\IE[\I{\{X \leq x\}}\,X]=\I{\{x \geq 1\}},
\end{gather*} 
hence $M \equiv 1$. Summarizing, the random vector ${\textbf Q}=(Q_1,\ldots,Q_d)$ can be simulated as follows:
\begin{itemize}
\item[(i)] Draw the random variable $Z\sim(1-\exp(-z))\,\nu_L(\mathrm{d}z)$.
\item[(ii)] Simulate $X_1,\ldots,X_d$ iid Bernoulli variables with success probability $1-\exp(-Z)$.
\item[(iii)] Draw a random variable $D$ which is uniformly distributed on $\{1,\ldots,d\}$.
\item[(iv)] Compute the random vector $(W_1,\ldots,W_d)$ as
\begin{gather*}
W_k:=\begin{cases}
X_k & \mbox{, if }k \neq D\\
1 & \mbox{, if }k=D\\
\end{cases}.
\end{gather*}
\item[(v)] Return ${\textbf Q}$, where $Q_k:=W_k/\big(\sum_{i=1}^{d}W_i\big)$, for $k=1,\ldots,d$.
\end{itemize}
\end{example}

\begin{example}[The case $L=N$]\label{example_standardPP2}
In the special case when $L=N$ is a standard Poisson process, one observes that $Z \equiv 1$ and Lemma~\ref{lemma_Pickands} boils down to a result in \cite[Lemma 4]{mai17}. It can thus be viewed as a generalization thereof.
\end{example}

\begin{example}[The case $F(x)=\min\{x,1\}$]\label{Example_German2}
The Pickands function $\ell_d$ is studied in \cite[Theorem~2]{bernhart15}. 
However, no simulation algorithm for $\mathbf{Q}$ has been found in that reference, a gap which we now fill. It is observed that
\begin{gather*}
\Psi_F(y) = \int_0^{1}\big(1-s^{y}\big)\,\mathrm{d}s = \frac{y}{y+1},\quad y \geq 0,
\end{gather*}
is the Bernstein function associated with a compound Poisson subordinator with unit exponential jumps and unit intensity. Hence, any $L$ leads to an admissible pair $(F,L)$ by Lemma \ref{lemma_simpleadmissible} below. In order to ensure $\Psi_H(1)=1$, $\nu_L$ must be normalized such that $\int_{(0,\infty]}z/(z+1)\,\nu_L(\mathrm{d}z)=1$. Second, for any fixed $z>0$ the distribution function $F(x)^z=\min\{x^z,1\}$ is trivial to simulate from via the inversion method, see \cite[p.\ 234]{maischerer17}. Furthermore, the distribution function of the random variable $M \sim x\,\mathrm{d}F^z(x)/\Psi_F(y)$ is given by $x \mapsto \min\{x^{z+1},1\}$, which is also easy to simulate from by the inversion method. Consequently, the simulation algorithm in Lemma~\ref{lemma_Pickands} is straightforward to implement, whenever the L\'evy subordinator $L$ is chosen such that the probability law of $Z$, that is 
\begin{gather*}
\IP(Z \in \mathrm{d}z) = \frac{z}{z+1}\,\nu_L(\mathrm{d}z),
\end{gather*}
can be simulated from. 
\end{example}

\begin{example}[The case $F(x)=\min\{\exp(x-1),1\}$]\label{example_anotherbernhart}
The Pickands function $\ell_d$ is computed in closed form in \cite[Theorem 1]{bernhart15}. It is given by 
\begin{align*}
\ell_d(x_1,\ldots,x_d)=\frac{d\,\Psi_H(d)}{\sum_{j=1}^d x_{[j]}^{-1}}-\sum_{i=1}^{d-1}\Big(\frac{d-i+1}{\sum_{j=i}^dx_{[j]}^{-1}}-\frac{d-i}{\sum_{j=i+1}^dx_{[j]}^{-1}}\Big)\,\Psi_H\big(d-i-\sum_{j=i+1}^d\frac{x_{[i]}}{x_{[j]}}\big),
\end{align*}
where $x_{[1]}\leq\ldots\leq x_{[d]}$ is the ordered list of $x_1,\ldots,x_d$. However, no simulation algorithm for $\mathbf{Q}$ has been found in that reference, a gap which we now fill. It is observed that
\begin{gather*}
\Psi_F(y) = \int_0^{1}\big(1-e^{y\,s-y}\big)\,\mathrm{d}s = 1-\frac{1-e^{-y}}{y},\quad y \geq 0,
\end{gather*}
is the Bernstein function associated with a compound Poisson subordinator with unit intensity and jumps that are uniformly distributed on $[0,1]$. Hence, any $L$ leads to an admissible pair $(F,L)$ by Lemma~\ref{lemma_simpleadmissible} below and \cite[Lemma~3]{bernhart15} shows that the Bernstein function $\Psi_H$ runs through all possible Bernstein functions when $L$ is varied. In order to ensure $\Psi_H(1)=1$, $\nu_L$ must be normalized such that $\int_{(0,\infty]}\big(1-(1-e^{-z})/z\big)\,\nu_L(\mathrm{d}z)=1$. Second, for any fixed $z>0$ the distribution function $F(x)^z=\min\{\exp(z\,x-z),1\}$ is easy to simulate from via the inversion method. Furthermore, the density $f_M$ of the random variable $M \sim x\,\mathrm{d}F^z(x)/\Psi_F(z)$ is computed to be
\begin{gather*}
f_M(x) = \frac{z^2\,e^{-z}}{z+e^{-z}-1}\,x\,e^{x\,z}\,\I{(0,1)}(x).
\end{gather*}
This density is bounded, so rejection-acceptance sampling with the uniform law on $[0,1]$ can be implemented to achieve an exact simulation scheme of $M$, see \cite[p.\ 235]{maischerer17}. Consequently, the simulation algorithm in Lemma~\ref{lemma_Pickands} is straightforward to implement, whenever the L\'evy subordinator $L$ is chosen such that the probability law of $Z$, that is 
\begin{gather*}
\IP(Z \in \mathrm{d}z) = \Big(1-\frac{1-e^{-z}}{z}\Big)\,\nu_L(\mathrm{d}z),
\end{gather*}
 can be simulated from. 
\end{example}

Figure~\ref{figtikz} schematically visualizes those admissible pairs $(F,L)$ for which either former literature or the present article provides knowledge about the associated extreme-value copula $C_d$. 

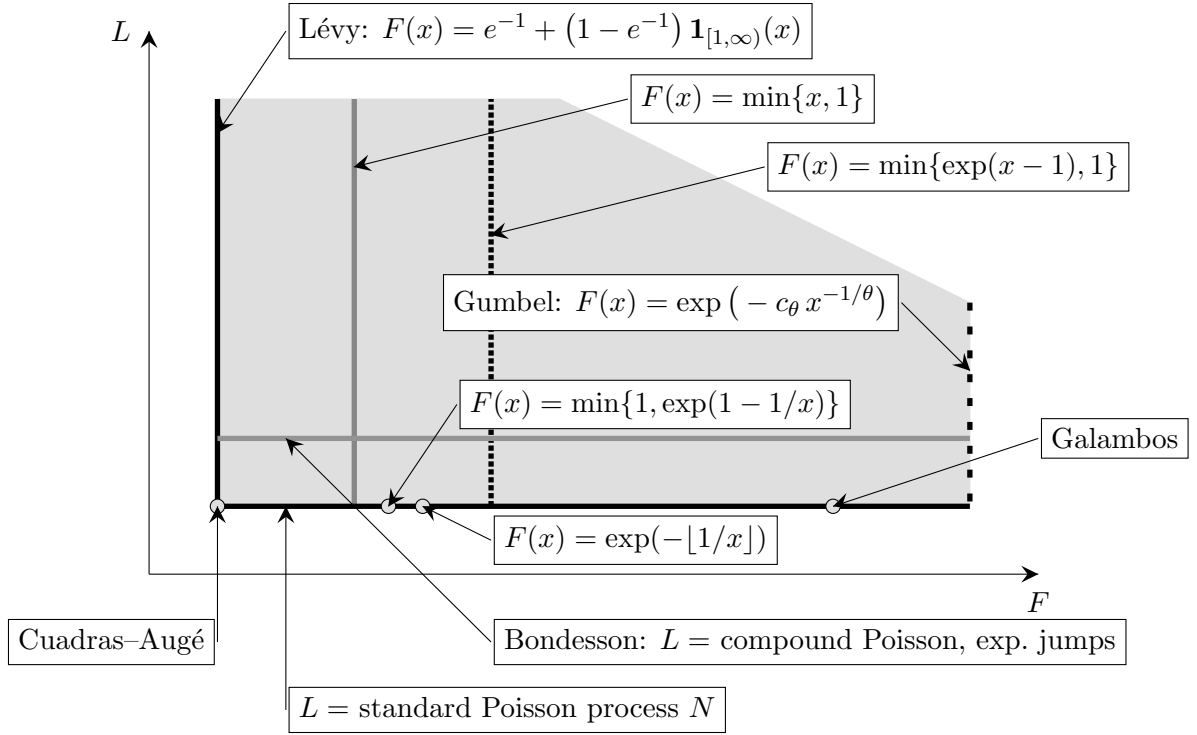
\begin{figure}[!ht]
\begin{tikzpicture}[scale=0.9, snake=zigzag]

\draw[directed] (0,0) -- (0,8);
\draw (0,8) node[left=3pt] {$ L $};
\draw (13,0) node[below=3pt] {$ F $};
\draw[directed] (0,0) -- (13,0);

\draw[gray!25, fill=gray!25] (1,1) -- (1,7) -- (6,7) -- (12,4) -- (12,1) -- cycle  ;

\draw[line width=0.7mm,black] (1,1) -- (1,7);
\draw[directed] (2,8) -- (1,6.5);
\draw (2,8) node[right=1pt,draw,rectangle] {L\'evy: $F(x)=e^{-1}+\big(1-e^{-1}\big)\,\I{[1,\infty)}(x)$};

\draw[line width=0.7mm,gray!95] (3,1) -- (3,7);
\draw[directed] (7,7) -- (3,6);
\draw (7,7) node[right=1pt,draw,rectangle] {$F(x)=\min\{x,1\}$};

\draw[line width=0.7mm,black, style=densely dotted] (5,1) -- (5,7);
\draw[directed] (9,6) -- (5,5);
\draw (9,6) node[right=1pt,draw,rectangle] {$F(x)=\min\{\exp(x-1),1\}$};

\draw[line width=0.7mm,gray!85] (1,2) -- (12,2);
\draw[directed] (5,-1) -- (2,2);
\draw (5,-1) node[right=1pt,draw,rectangle] {Bondesson: $L=$ compound Poisson, exp.\ jumps};

\draw[line width=0.7mm,black] (1,1) -- (12,1);
\draw[directed] (2,-2) -- (2,1);
\draw (2,-2) node[right=1pt,draw,rectangle] {$L=$ standard Poisson process $N$};

\draw[line width=0.7mm,black, loosely dashed] (12,4) -- (12,1);
\draw[directed] (11,4) -- (12,3);
\draw (11,4) node[left=1pt,draw,rectangle,fill=white] {Gumbel: $F(x)=\exp\big(-c_{\theta}\,x^{-{1/\theta}}\big)$};

\draw[fill=gray!25] (10,1) circle (3pt);
\draw[directed] (13,2) -- (10,1);
\draw (13,2) node[right=1pt,draw,rectangle] {Galambos};

\draw[fill=gray!25] (1,1) circle (3pt);
\draw[directed] (1,-1) -- (1,1);
\draw (1,-1) node[left=1pt,draw,rectangle] {Cuadras--Aug\'e};

\draw[fill=gray!25] (4,1) circle (3pt);
\draw[directed] (5,0.5) -- (4,1);
\draw (5,0.5) node[right=1pt,draw,rectangle] {$F(x)=\exp(-\lfloor{1/x}\rfloor)$};

\draw[fill=gray!25] (3.5,1) circle (3pt);
\draw[directed] (4.5,2.5) -- (3.5,1);
\draw (4.5,2.5) node[right=1pt,draw,rectangle,fill=white] {$F(x) = \min\{1,\exp(1-1/x)\}$};



\end{tikzpicture}
\caption{Illustration of well-understood pairs $(F,L)$. Vertical and horizontal lines of the same greyscale and line-type combination are complementary in the sense of the distributional equality (\ref{interestingequaldist}) below. The L\'evy case is covered in Example~\ref{Example_Levy_case} and originates from \cite{maischerer09}. The case $F(x)=\min\{x,1\}$ relates to Example~\ref{Example_German2}, see \cite{bernhart15} for details. The case $F(x)=\min\{\exp(x-1),1\}$ relates to Example~\ref{example_anotherbernhart} and \cite{bernhart15}. The Gumbel case is covered in Examples~\ref{ex_Gumbel} and \ref{Example_Gumbel}, see also \cite{mai17}, the name stems from \cite{gumbel60,gumbel61}. The Galambos and the Cuadras--Aug\'e copulas are named after \cite{galambos75} and \cite{cuadras81}, respectively. The case of a standard Poisson process is discussed in Example~\ref{example_standardPP2} and \cite{mai17}. If, additionally, $F(x)=\exp(-\lfloor{1/x}\rfloor)$ or $F(x) = \min\{1,\exp(1-1/x)\}$, this yields \cite[Example 5.1]{molchanov18} or \cite[Example 4.6]{molchanov18}, respectively. The Bondesson family relates to Lemma~\ref{lemma_bondesson}.}
\label{figtikz}
\end{figure}

\section{Application 2: Series representations for infinitely divisible laws}\label{sec_ID}
We recall from Equations~(\ref{Psi_F_A}) and (\ref{anotherrepr_PsiH}), using Lemma~\ref{BF_mai18}, that the law of the random variable $H_1$ is infinitely divisible on $[0,\infty]$ with associated Bernstein function
\begin{gather}\label{Psi_H_B}
\Psi_H(x) = \int_{(0,\infty]}\int_{(0,\infty]}\big( 1-e^{-x\,s\,u}\big) \,\nu_F(\mathrm{d}s)\,\nu_L(\mathrm{d}u),\quad x \geq 0.
\end{gather} 
This double integral representation indicates that the roles of $F$ and $L$ can be switched without changing the one-dimensional marginal distribution of $H_1$. More precisely,
\begin{gather}
\int_{0}^{\infty}-\log\big( F(s-)\big)\,\mathrm{d}L_s \stackrel{d}{=} \int_{0}^{\infty}-\log\big( F^{(L)}(s-)\big)\,\mathrm{d}L^{(F)}_s,
\label{interestingequaldist}
\end{gather}
where for a given L\'evy subordinator $L$ the function $F^{(L)} \in \F$ is uniquely determined by the equality $\Psi_{F^{(L)}}=\Psi_L$, and for a given $F\in \F$ the L\'evy subordinator $L^{(F)}$ is uniquely determined by the equality $\Psi_{L^{(F)}}=\Psi_F$. Recall that the bijection between the set of L\'evy measures on $(0,\infty]$ and $\F$ is explicitly stated in Lemma~\ref{BF_mai18}. 
\par
Since admissibility by definition means that the law of $H_1$ is non-trivial, one consequence of this ``duality'' is that $(F,L)$ is admissible if and only if $(F^{(L)},L^{(F)})$ is admissible, which implies the following simple admissibility criterion, that could alternatively also be proved directly.

\begin{lemma}[Simple admissibility criterion]\label{lemma_simpleadmissible}
Let $F \in \F$ and $L$ a L\'evy subordinator without drift. If $\Psi_F$ is bounded and satisfies $\lim_{x \downarrow 0}\Psi_F^{'}(x)<\infty$, the pair $(F,L)$ is admissible, no matter how $L$ is chosen. 
\end{lemma}
\begin{proof}
See the Appendix.
\end{proof}

The integral definition of $H_1$ becomes an infinite series whenever the integrator $L$ is a (driftless) compound Poisson process, i.e.\ if $\Psi_L$ is bounded. Switching the roles of $L$ and $F$ according to the duality (\ref{interestingequaldist}), we also obtain a series representation for $H_1$ if $L^{(F)}$ is a (driftless) compound Poisson subordinator, i.e.\ if $\Psi_F$ is bounded. The following example, based on an idea originally due to \cite{bondesson82}, illustrates how this can be useful. We denote by ID$[0,\infty]$ the set of all infinitely divisible laws on $[0,\infty]$.

\begin{example}[Series representations for ID{$[0,\infty]$} from duality] \label{ex_duality}
Reconsidering Example~\ref{Example_Levy_case}, on the left-hand side of (\ref{interestingequaldist}) let $F(x)=\exp(-1)+(1-\exp(-1))\,\I{[1,\infty)}(x)$, for $x\geq 0$,  such that $\Psi_F(x)=1-\exp(-x)$ and $L^{(F)}=N$ is a Poisson process with unit intensity, whose unit exponential inter-arrival times we denote by $\{\epsilon_k\}_{k \geq 1}$ and the associated jump times are $\{\tau_k\}_{k\geq 1}$. Then (\ref{interestingequaldist}) becomes
\begin{gather*}
L_1 = \int_{0}^{\infty}-\log\big( F(s-)\big)\,\mathrm{d}L_s \stackrel{d}{=} \sum_{k \geq 1}S_{\nu_L}^{-1}(\underbrace{\epsilon_1+\ldots+\epsilon_k}_{\tau_k})\,\I{\{\underbrace{\epsilon_1+\ldots+\epsilon_k}_{\tau_k} \leq \nu_L((0,\infty])\}},
\end{gather*}
where $S_{\nu_L}(t):=\nu_L\big( (t,\infty]\big)$ is the survival function of the L\'evy measure $\nu_L$ of $L$, and $S_{\nu_L}^{-1}$ its generalized inverse. In particular, this series is almost surely finite in case of a compound Poisson distribution, i.e.\ if $\Psi_L$ is bounded (resp.\ $\nu_L$ is finite). This is more or less the representation of an infinitely divisible law in terms of a series representation involving only independent exponentials that \cite{bondesson82} proposes as a basis for his simulation ansatz (restricted to laws on $[0,\infty]$ in the present context). It is of particular use in those cases where the survival function of the L\'evy measure has an inverse $S_{\nu_L}^{-1}$ in closed form. If, in addition, the L\'evy measure is finite, i.e.\ one has a compound Poisson distribution, one obtains an exact simulation algorithm. When simulating this compound Poisson law, the representation is useful particularly if we have no simulation algorithm for the jump size distribution at hand, but we are able to compute the inverse of the Laplace transform of the jump size distribution in closed form.  
\end{example}

\begin{remark}[Alternative series representation for ID{$[0,\infty]$}]
Example \ref{example_LS} shows that Lemma \ref{lemma_repr}(b) provides a series representation for ID$[0,\infty]$ that is different from that in Example \ref{ex_duality}, namely
\begin{gather*}
L_1 \stackrel{d}{=} \sum_{k \geq 1}Z_k\,\I{\big\{(\epsilon_1+\ldots+\epsilon_k)\,\big(1-e^{-Z_k}\big) \leq  \Psi_L(1)\big\}},
\end{gather*}
where $\{Z_k\}_{k \geq 1}$ is an iid sequence distributed according to the probability measure $(1-e^{-z})\,\nu_L(\mathrm{d}z)/\Psi_L(1)$, independent of $\{\epsilon_k\}_{k \geq 1}$. This representation is always an infinite series, even when $L_1$ has a compound Poisson distribution. However, if a closed form of $S_{\nu_L}^{-1}$ is unknown but a simulation algorithm for the random variables $Z_k$ is known, this series representation might be preferred over the one in Example~\ref{ex_duality}.
\end{remark}


The series representation in Example \ref{ex_duality} was based on the choice $L=N$ of a Poisson process as integrator. In the sequel, we choose as integrator a compound Poisson process with unit exponential jumps. In this case, Lemma~\ref{lemma_bondesson} below shows that the law of $H_1$ lies in the so-called \emph{Bondesson family} BO$(0,\infty)$, see \cite{bondesson81}. This means that the survival function of the L\'evy measure equals the Laplace transform of a measure $\rho$ on $(0,\infty)$, called the \emph{Stieltjes measure}. From an analytical viewpoint, the Bernstein function $\Psi_H$, which is said to be \emph{complete} in this case, can be represented as
\begin{gather*}
\Psi_H(x) = \int_0^{\infty}\frac{x}{x+u}\,\rho(\mathrm{d}u),\quad x \geq 0,
\end{gather*}
with the Stieltjes measure $\rho$ on $(0,\infty)$ satisfying $\int_0^{\infty}(1+u)^{-1}\,\rho(\mathrm{d}u)<\infty$. There are examples for laws on BO$(0,\infty)$ for which the Stieltjes measure $\rho$ is way more convenient to handle than the associated L\'evy measure\footnote{See, e.g., Families 3, 5, 27, 28, 29, 33, 35, 45, 46, 63, 64, 88, and 89 in the list of complete Bernstein functions of \cite[Chapter 15]{schilling10}.}, we provide some examples below. For these cases, Lemma~\ref{lemma_bondesson} below provides a series representation in a similar spirit as that for ID$[0,\infty]$ in Example~\ref{ex_duality}.
\par
We fix $L$ as a compound Poisson subordinator with unit exponential jumps and unit intensity, i.e.\ $\Psi_{L}(x)=x/(x+1)$. Furthermore, we let $F \in \F$ be arbitrary, and only assume that the left-end point of the support of $X \sim F$ equals zero, which by Lemma~\ref{BF_mai18} is equivalent to postulating $\nu_F(\{\infty\})=0$, i.e.\ $\Psi_F$ has no killing. In this situation, the measure
\begin{gather}
\rho_F(A):=\nu_F\Big(\Big\{\frac{1}{u}\,:\,u \in A \Big\} \Big)
\label{stieltjesbiject}
\end{gather}
is well-defined for all Borel sets $A \in \mathcal{B}\big( (0,\infty)\big)$, and we observe that
\begin{gather*}
\int_0^{\infty}\frac{1}{1+t}\,\rho_F(\mathrm{d}t)=\int_0^{\infty}\frac{u}{1+u}\,\nu_F(\mathrm{d}u) \leq \int_0^{\infty}\min\{u,1\}\,\nu_F(\mathrm{d}u) < \infty,
\end{gather*}
so $\rho_F$ is a proper Stieltjes measure. To simplify notation, we denote
\begin{gather*}
\hat{\F}:=\{F \in \F\,:\,\nu_F(\{\infty\})=0\}
\end{gather*}
and obtain the following result. The series representation in part (c) is a special case of the series representation discussed in \cite[p.\ 862]{bondesson82}.
\begin{lemma}[Series representation in terms of the Stieltjes measure]\label{lemma_bondesson}
Let $F \in \hat{\F}$ and $L$ a compound Poisson subordinator with unit intensity and unit exponentially distributed jumps, i.e.\ $\Psi_L(x)=x/(x+1)$.
\begin{itemize}
\item[(a)] The law of $H_1$ is in $\text{BO}(0,\infty)\subset \text{ID}(0,\infty)$ with associated Stieltjes measure $\rho_F$.
\item[(b)] The mapping $F \mapsto \rho_F$, defined by (\ref{stieltjesbiject}) and Lemma~\ref{BF_mai18}, defines a bijection between $\hat{\F}$ and BO$(0,\infty)$. The inverse mapping $\rho \mapsto F_{\rho}$ is
\begin{gather*}
F_{\rho}(x) = \begin{cases}
e^{-g_{\rho}^{-1}(x)} &\mbox{, if }x<\rho\big( (0,\infty)\big)\\
1 & \mbox{, else}\\
\end{cases},\quad g_{\rho}(x):=\rho\big( (0,1/x)\big),
\end{gather*}
where $g^{-1}_{\rho}$ denotes the generalized inverse of the non-increasing function $g_{\rho}$.
\item[(c)] We have the following equality in law, with $\{\epsilon_k,\,J_k\}_{k \geq 1}$ iid unit exponentials:
\begin{gather*}
H_1 \stackrel{d}{=}\sum_{k \geq 1} J_k\,g^{-1}_{\rho}(\epsilon_1+\ldots+\epsilon_k)\,\I{\{\epsilon_1+\ldots+\epsilon_k \leq \rho( (0,\infty))\}}.
\end{gather*} 
\end{itemize}
\end{lemma}
\begin{proof}
See the Appendix.
\end{proof}

Like the series representation in Example \ref{ex_duality} for ID$[0,\infty]$ is useful if the L\'evy measure is nice, the representation in Lemma~\ref{lemma_bondesson}(c) can be used to construct simulation algorithms for infinitely divisible distributions from the Bondesson family, when the Stieltjes measure $\rho$ is nice. In case of a compound Poisson distribution the series is even finite, hence the simulation algorithm is exact. We provide some examples to demonstrate this procedure. 

\begin{example}[Exact simulation of some compound Poisson laws]\label{Ex_CPL}
Let $\rho$ be some finite measure on $(0,\infty)$, hence it automatically is a Stieltjes measure. We denote by $X$ a generic random variable in BO$(0,\infty)$ associated with this Stieltjes measure. There exists a distribution function $G$ of some non-negative random variable with $G(0)=0$ and a constant $\beta:=\rho\big( (0,\infty)\big) >0$ such that $\rho\big( (0,x]\big)=\beta\,G(x)$. The L\'evy measure $\nu$ associated with this Stieltjes measure $\rho$ is determined by its survival function, which satisfies $\nu\big( (x,\infty)\big)=\beta\,\varphi_G(x)$, where $\varphi_G(x):=\beta\,\int_0^{\infty}\exp(-x\,u)\,\mathrm{d}G(u)$ denotes the Laplace transform of $\mathrm{d}G$. In this case, we have that
\begin{gather*}
g_{\rho}^{-1}(y) = \frac{1}{G^{-1}(y/\beta)},\quad 0 <y<\beta,
\end{gather*}
where $G^{-1}$ is the generalized inverse of the distribution function $G$. Consequently, according to Lemma \ref{lemma_bondesson}(c) the compound Poisson distribution $X$ with intensity $\beta$ and jump size density $-\varphi_G^{'}$ has the finite series representation
\begin{gather}
X \stackrel{d}{=} \sum_{k \geq 1}\frac{J_k}{G^{-1}\big((\epsilon_1+\ldots+\epsilon_k)/\beta\big)}\,\I{\{\epsilon_1+\ldots+\epsilon_k \leq \beta\}},
\label{CPP_seriesrepr}
\end{gather}
with $\{\epsilon_k,J_k\}_{k \geq 1}$ iid unit exponentials. It is not difficult to come up with examples for $G$ such that $G^{-1}$ is in closed form, but neither is a simulation algorithm for the density $-\varphi_G^{'}$ at hand, nor is the inverse of $\varphi_G$ in closed form. In such a situation, (\ref{CPP_seriesrepr}) provides a convenient basis for an exact simulation algorithm. One particular example is given by Family 45 in \cite[Chapter~15]{schilling10}: we have $\beta = 0.5$ and $G(x)=\min\{x,1\}\,(2-\min\{x,1\})$  and obtain $G^{-1}(x)=1-\sqrt{1-x}$ for $x \in (0,1)$, leading to
\begin{gather*}
X \stackrel{d}{=}\sum_{k \geq 1}\frac{J_k\,\I{\{\epsilon_1+\ldots+\epsilon_k \leq 0.5\}}}{1-\sqrt{1-2\,(\epsilon_1+\ldots+\epsilon_k)}}.
\end{gather*}
The associated Bernstein function is given by $\Psi(x)=x\,(1+x)\,\log(1+1/x)-x$. 
\end{example}

\begin{example}[A few non-compound Poisson examples]\label{more_examples}
With a parameter $\theta>0$, we consider the complete, unbounded Bernstein functions
\begin{gather*}
\Psi_1(x) = \frac{x}{\sqrt{x+\theta}},\quad \Psi_2(x)=\frac{x}{\theta-x}\,\log\Big( \frac{\theta}{x}\Big),\quad \Psi_3(x)=\sqrt{x}\,\arctan\Big( \sqrt{\frac{x}{\theta}}\Big),
\end{gather*}
which correspond to Families 5, 33, and 64 in \cite[Chapter 15]{schilling10}. For each of them, the associated Stieltjes measure has a convenient form, namely
\begin{gather*}
\rho_1(\mathrm{d}x) = \frac{\I{\{x>\theta\}}\,\mathrm{d}x}{\pi\,\sqrt{x-\theta}},\quad \rho_2(\mathrm{d}x)=\frac{\mathrm{d}x}{\theta+x},\quad \rho_3(\mathrm{d}x)=\frac{\I{\{x>\theta\}}\,\mathrm{d}x}{2\,\sqrt{x}},
\end{gather*}
and the associated function $g_{\rho}^{-1}$ from Lemma \ref{lemma_bondesson} can be computed in closed form, to wit
\begin{gather*}
g_{\rho_1}^{-1}(x)=\frac{1}{\theta+(\pi\,x/2)^2},\quad g_{\rho_2}^{-1}(x)=\frac{1}{\theta\,(e^{x}-1)},\quad g_{\rho_3}^{-1}(x)=\frac{1}{(x+\sqrt{\theta})^2}.
\end{gather*}
\end{example}

\par
Figure~\ref{simulation_outcomes} presents, based on $100.000$ observations, the empirical and theoretical Bernstein function in all of the four cases. 

\begin{figure}[!ht]
\includegraphics[width=13cm]{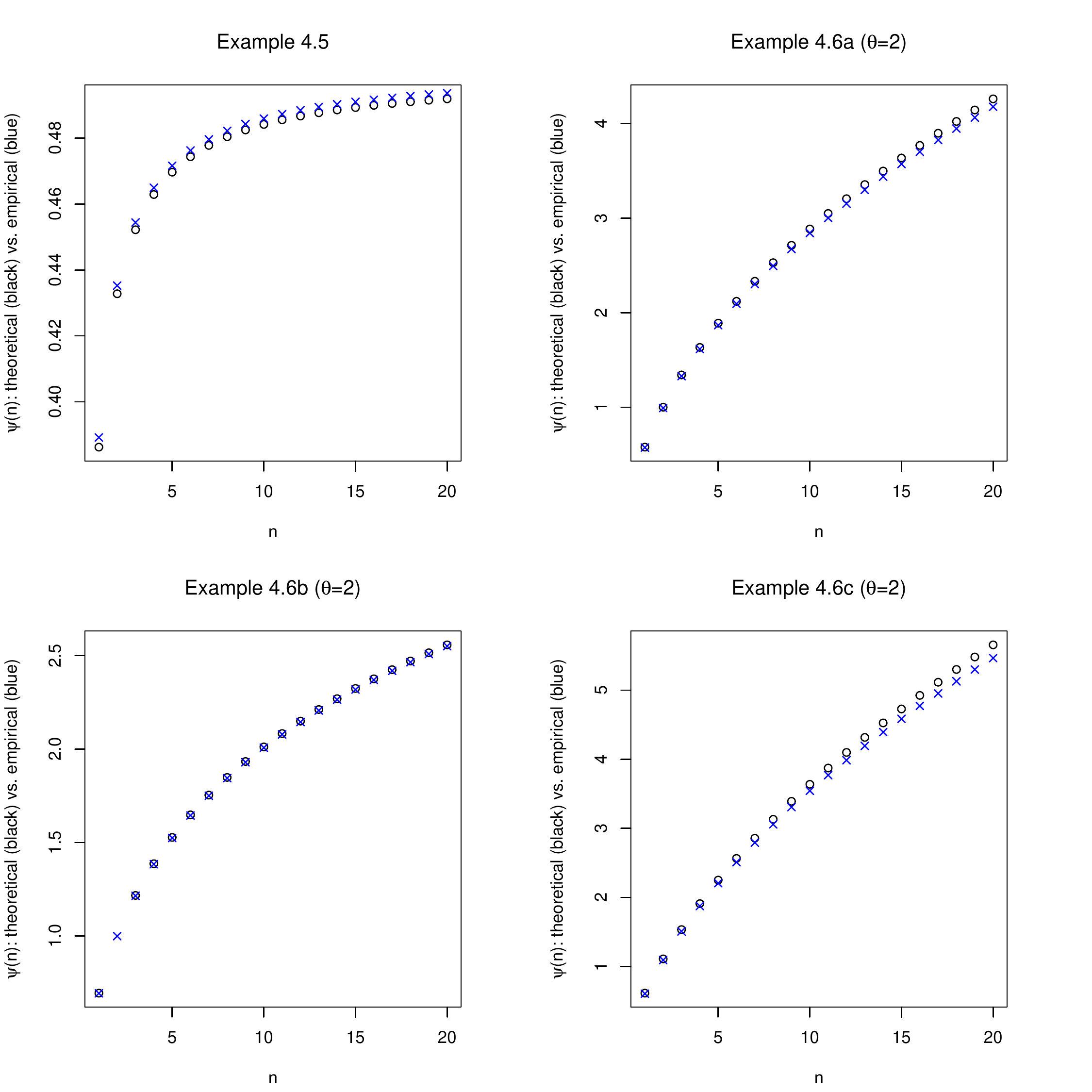}
\caption{Empirical test of the validity of the sampling routines from Examples~\ref{Ex_CPL} and \ref{more_examples}. The theoretical Bernstein functions evaluated at $\{1,\ldots,20\}$ (black circles) are compared to the estimated ones (blue crosses). The \textsf{R}-code to simulate from the described families is available upon request by the authors.}\label{simulation_outcomes}
\end{figure}

\section{Conclusion}\label{sec_conc}
We have studied a semi-parametric family of non-decreasing stochastic processes $H=\{H_t\}_{t \geq 0}$, which comprises (possibly killed) L\'evy subordinators without drift as a special case. Whereas a L\'evy subordinator $L$ is conveniently specified by a Bernstein function $\Psi_L$, the process $H$ is specified by a pair $(\Psi_F,\Psi_L)$ of two Bernstein functions. From a theoretical point of view, we have demonstrated that the parameterization in terms of a distribution function $F$ and a L\'evy subordinator $L$ provides a convenient apparatus to study distributional properties of $H$. In particular, we have established a canonical series representation and have studied the natural filtration of $H$, highlighting that the independent increment property is exclusive to the L\'evy subordinator subfamily. From a practical point of view, we have explained how to simulate exactly and accurately the $d$-variate extreme-value copula associated with each process $H$. Furthermore, we have used the derived setting to establish a series representation for infinitely divisible laws from the Bondesson family that is given in terms of its Stieltjes measure only.

\section*{Appendix: Proofs}

\begin{proof}[of Lemma \ref{STDF}]
Writing out the Riemann--Stieltjes definition of $H$, introducing for $N \in \IN$ and $R \gg 1$ the notation
\begin{gather*}
s_{n,N}^{(R)}:=\frac{1}{R}+\frac{n\,\big(R-\frac{1}{R}\big)}{N},\quad n=0,\ldots,N,
\end{gather*}
we have
\begin{align*}
H_t&=\lim_{R \rightarrow \infty}\int_{\frac{1}{R}}^{R}-\log\Big[ F\Big(\frac{s}{t}-\Big)\Big]\,\mathrm{d}L_s\\
&=\lim_{R \rightarrow \infty}\lim_{N\rightarrow \infty}\sum_{n=1}^{N}-\log\Big[ F\Big(\frac{s_{n-1,N}^{(R)}}{t}-\Big)\Big]\,\big(L_{s_{n,N}^{(R)}}- L_{s_{n-1,N}^{(R)}}\big).
\end{align*}
Using the bounded convergence theorem in $(\ast)$ and the stationary and independent increment property of $L$ in $(\ast\ast)$ implies
\begin{align*}
&\hspace{-0.5cm}\IE\Big[ e^{-\sum_{k=1}^{d}H_{t_k}}\Big] \\
& = \IE\Big[ \exp\Big\{-\lim_{R \rightarrow \infty}\lim_{N\rightarrow \infty}\sum_{n=1}^{N}\big( L_{s_{n,N}^{(R)}}- L_{s_{n-1,N}^{(R)}}\big)\,\sum_{k=1}^{d}-\log\Big[ F\Big(\frac{s_{n-1,N}^{(R)}}{t_k}-\Big)\Big] \Big\}\Big]\\
& \stackrel{(\ast)}{=} \lim_{R \rightarrow \infty}\lim_{N\rightarrow \infty}\IE\Big[ \exp\Big\{-\sum_{n=1}^{N}\big( L_{s_{n,N}^{(R)}}- L_{s_{n-1,N}^{(R)}}\big)\,\sum_{k=1}^{d}-\log\Big[ F\Big(\frac{s_{n-1,N}^{(R)}}{t_k}-\Big)\Big] \Big\}\Big]\\
& \stackrel{(\ast\ast)}{=} \lim_{R \rightarrow \infty}\lim_{N\rightarrow \infty}\prod_{n=1}^{N}\IE\Big[ \exp\Big\{- L_{\frac{R-\frac{1}{R}}{N}}\,\sum_{k=1}^{d}-\log\Big[ F\Big(\frac{s_{n-1,N}^{(R)}}{t_k}-\Big)\Big] \Big\}\Big]\\
& = \lim_{R \rightarrow \infty}\lim_{N\rightarrow \infty}\prod_{n=1}^{N}\exp\Big\{-\frac{R-\frac{1}{R}}{N}\,\int_{(0,\infty]}\Big(1-\prod_{k=1}^{d} F\Big(\frac{s_{n-1,N}^{(R)}}{t_k}-\Big)^{u} \Big)\,\nu_L(\mathrm{d}u)\Big\}\\
& = \lim_{R \rightarrow \infty}\exp\Big\{-\int_{\frac{1}{R}}^{R}\,\int_{(0,\infty]}\Big(1-\prod_{k=1}^{d} F\Big(\frac{s}{t_k}-\Big)^{u} \Big)\,\nu_L(\mathrm{d}u)\,\mathrm{d}s\Big\}\\
& = \exp\Big\{-\int_0^{\infty}\,\int_{(0,\infty]}\Big(1-\prod_{k=1}^{d} F\Big(\frac{s}{t_k}-\Big)^{u} \Big)\,\nu_L(\mathrm{d}u)\,\mathrm{d}s\Big\}.
\end{align*}
The order of the two remaining integrations can be switched by Tonelli's Theorem. When integrating with respect to $\mathrm{d}s$, it is further possible to change from $F(s-)$ to $F(s)$, since the at most countably many jump times of $F$ play no role in the integration. This finishes the proof.
\end{proof}

\begin{proof}[of Lemma \ref{lemma_technicality}]
The equivalence of $(a)$ and $(b)$ is obvious, as well as the fact that $(c)$ implies (a) and (b). The only non-obvious statement is that admissibility implies (c). Denote the biggest argument by $t_{[d]}:=\max\{t_1,\ldots,t_d\}$. From the probabilistic viewpoint the statement is pretty obvious, since admissibility means that $\IP(H_{t}<\infty)>0$ for all $t \geq 0$. But this implies that 
\begin{gather*}
\IP\Big( \sum_{j=1}^{d}H_{t_j}< \infty\Big)=\IP\Big( H_{t_{[d]}}< \infty\Big)>0,
\end{gather*}
since $t\mapsto H_t$ is non-decreasing, which implies the claim. Alternatively, we can argue fully analytically by computing
\begin{align*}
& \ell(t_1,\ldots,t_d,0,\ldots) \leq \int_{(0,\infty]} \int_0^{\infty}\Big(1-F\big( \frac{s}{t_{[d]}}\big)^{d\,y}\Big)\,\mathrm{d}s\,\nu_L(\mathrm{d}y) \\
& \qquad =t_{[d]}\,\int_{(0,\infty]}\Psi_F(d\,y)\,\nu_L(\mathrm{d}y) \leq t_{[d]}\,d\,\int_{(0,\infty]}\Psi_F(y)\,\nu_L(\mathrm{d}y) \stackrel{(\ast)}{<}\infty,
\end{align*}
where $(\ast)$ follows from ($L$-)admissibility and the above inequality from the estimate $g(d\,y) \leq d\,g(y)$, which holds for any concave function $g$ on $[0,\infty)$ with $g(0)=0$ and $d \in \IN$, such as $g=\Psi_F$. To see this, for $d \in \IN$, concavity and $g(0)=0$ imply that
\begin{gather*}
d\,g(y) = d\,g\Big( \frac{d-1}{d}\cdot 0+\frac{1}{d}\,d\,y\Big) \geq d\,\Big(\frac{d-1}{d}\,g(0)+\frac{1}{d}\,g(d\,y)\Big)=g(d\,y).
\end{gather*}
\end{proof}

\begin{proof}[of Lemma \ref{lemma_repr}]
As an auxiliary step, we first show that
\begin{align}
0=\lim_{x\searrow 0 }\log\big(F(x)\big)\,L_x=\lim_{x\nearrow \infty }\log\big(F(x)\big)\,L_x.
\label{partaproof}
\end{align}
To this end, we observe that 
\begin{align}\label{usefulinequality}
\Psi_F(u\,t)\leq \max\{u,1\}\,\Psi_F(t),\quad \forall u,t\geq 0,
\end{align}
which follows from $\Psi_F(x)=\int_0^\infty \big(1-e^{-x\,u}\big)\,\nu_F(\mathrm{d}u)$ and $(1-e^{-t\,u\,x})\leq \max\{u,1\}(1-e^{-t\,x})$.
\par 
Second, we find for $x>0$ the estimate
\begin{align}\label{usefulinequality2}
\nonumber x\,\big(1-F(x)^{u\,t}\big)&=\int_0^x \big(1-F(x)^{u\,t} \big)\,\mathrm{d}s\\
&\leq \int_0^x \big(1-F(s)^{u\,t} \big)\,\mathrm{d}s \leq \Psi_F(u\,t)\stackrel{(\ref{usefulinequality})}{\leq} \max\{u,1\}\, \Psi_F(t).
\end{align}
Now we use dominated convergence, justified by the admissibility condition $F \in \F_L$ and Equation (\ref{usefulinequality2}), to find
\begin{align}\label{usefulinequality3}
\lim_{x\rightarrow \ast}\int_0^\infty x\,\big(1-F(x)^{u\,t}\big)\nu_L(\mathrm{d}t)=\int_0^\infty \lim_{x\rightarrow \ast} x\,\big(1-F(x)^{u\,t}\big)\nu_L(\mathrm{d}t)=0,\quad \ast\in\{0,\infty\}.
\end{align}
The limit inside the integral being zero is obvious for $\ast=0$. For $\ast=\infty$ we use the following argument: Consider a random variable $Y$ with $\IP(Y\leq s)=F(s)^{u\,t}$, for $s\geq 0$ and fixed $u,t\geq 0$. Again we use dominated convergence, justified by $x\,\I{\{Y>x\}}\leq Y$ and 
\begin{align*}
\IE[Y]=\int_0^\infty \IP(Y>x)\,\mathrm{d}x=\int_0^\infty \big(1-F(x)^{u\,t}\big)\,\mathrm{d}x=\Psi_F(u\,t)<\infty
\end{align*}
to find 
\begin{align*}
\lim_{x\rightarrow\infty} x\,\big(1-F(x)^{u\,t}\big)=\lim_{x\rightarrow\infty} x\,\IE[\I{\{Y>x\}}]= \IE[\lim_{x\rightarrow\infty} x\,\I{\{Y>x\}}]=0.
\end{align*}
Finally, the random variable $-\log(F(x))\,L_x$ is infinitely divisible and its Bernstein function is found to be $\int_0^\infty x\,\big(1-F(x)^{u\,t}\big)\nu_L(\mathrm{d}t)$. In Equation~(\ref{usefulinequality3}) we have shown that this tends to zero as $x\rightarrow\ast$, for all $u>0$, which establishes the claimed identity (\ref{partaproof}).
\begin{itemize}
\item[(a)] We use the integration-by-parts formula from \cite[Theorem~21.67, p.~419]{hewitt69} and (\ref{partaproof}) to find
\begin{align*}
H_t&=\int_0^\infty -\log\big(F(\frac{s}{t}-)\big)\,\mathrm{d}L_s=\lim_{R\rightarrow\infty}\int_{1/R}^R -\log\big(F(\frac{s}{t}-)\big)\,\mathrm{d}L_s\\
&=\lim_{R\rightarrow\infty}\Big\{ \int_{1/R}^R L_s\, \mathrm{d}\Big(\log\big(F(\frac{s}{t})\big)\Big) \\
&\qquad +L_R\,\Big(-\log\big(F(\frac{R}{t})\big)\Big)- L_{1/R-}\,\Big(-\log\big(F(\frac{1}{R\,t}-)\big)\Big)   \Big\}\\
&=\int_0^\infty L_{s} \mathrm{d}\big(\log(F(\frac{s}{t})\big)+0+0=\int_0^\infty L_{s\, t} \mathrm{d}\big(\log(F(s))\big),\quad t>0.
\end{align*}
\item[(b)] Denoting by $\delta_{(x,z)}$ the Dirac measure at a point $(x,z)$ in the plane, we note that $P:=\sum_{k \geq 1}\delta_{(\epsilon_1+\ldots+\epsilon_k,Z_k)}$ is a Poisson random measure on $[0,\infty)\times (0,\infty]$ with mean measure $\mathrm{d}x \times \big( \frac{\Psi_F(z)}{\Psi_H(1)}\,\nu_L(\mathrm{d}z) \big)$ by \cite[Proposition 3.8]{resnick87}. Consequently, denoting the stochastic process on the right-hand side of the claim in statement (b) by $\{\tilde{H}_t\}_{t \geq 0}$, the Laplace functional formula for Poisson random measure \cite[Proposition 3.6]{resnick87} yields
\begin{align*}
\IE\Big[ e^{-\sum_{j=1}^{d}\tilde{H}_{t_j}}\Big] &= \exp\Big( -\int_{(0,\infty]} \int_{0}^{\infty} \Big(1-\prod_{j=1}^{d}F\Big( \frac{x}{t_j}-\Big)^{z}\Big)\,\mathrm{d}x\,\nu_L(\mathrm{d}z)\Big)\\
& = e^{-\ell(t_1,\ldots,t_d,0,0,\ldots)}=\IE\Big[ e^{-\sum_{j=1}^{d}{H}_{t_j}}\Big]
\end{align*}
for arbitrary $t_1,\ldots,t_d \geq 0$ and $d \in \IN$, establishing the claim.
\end{itemize}
\end{proof}

\begin{proof}[of Lemma \ref{lemma_CP}]
First of all, $\Psi_H$ may be represented as
\begin{gather}
\Psi_H(x) = \int_0^{\infty}\Psi_L\Big( -\log\big( F(s)^{x}\big)\Big)\,\mathrm{d}s,
\label{onerepr_PsiH}
\end{gather}
which follows from (\ref{anotherrepr_PsiH}) with the help of Tonelli's Theorem. The sufficiency of the conditions (a) and (b) is clear from this representation, since $\Psi_L$ is bounded and the improper integral $\int_0^{\infty}[\ldots]\, \mathrm{d}s$ is actually a finite integral $\int_0^{u_F}[\ldots]\, \mathrm{d}s$, where $u_F$ is the right-end point of the support of $\mathrm{d}F$. Necessity is more difficult to observe. To this end, assume that $\Psi_{L}$ is unbounded. There is an $\epsilon>0$ and a $\delta>0$ such that $f(z):=-\log(F(z))>\delta$ for $z\in(0,\epsilon)$ by the condition $\IE[X]>0$ in the definition of $\F$. Consequently,
\begin{gather*}
\Psi_H(u) \geq \int_0^{\epsilon}\Psi_{L}(u\,\delta)\,\mathrm{d}s = \Psi_{L}(u\,\delta)\,\epsilon
\end{gather*} 
is obviously unbounded. Hence, we already see that $\Psi_{L}$ needs to be bounded. In this case, $L$ is of compound Poisson type, i.e.\ there is a Laplace transform $\varphi$ of a positive random variable on $(0,\infty]$ and a number $\beta>0$ such that $\Psi_{L}=\beta\,(1-\varphi)$. Consequently, we observe from (\ref{onerepr_PsiH}) that
\begin{gather*}
\Psi_H(u) = \beta\,\int_0^{\infty}\Big(1-\varphi\big( u\,f(s)\big)\Big)\,\mathrm{d}s.
\end{gather*} 
By assumption, we know that $\Psi_H$ is bounded, so we see with the help of Fatou's Lemma that 
\begin{align*}
\infty &> \lim_{u \rightarrow \infty}\Psi_H(u) = \lim_{u \rightarrow \infty}\beta\,\int_0^{\infty}\Big(1-\varphi\big( u\,f(s)\big)\Big)\,\mathrm{d}s\\
& \geq \beta\,\int_0^{\infty}\Big(1-\liminf_{u \rightarrow \infty}\varphi\big( u\,f(s)\big)\Big)\,\mathrm{d}s.
\end{align*}
But $\varphi$ tends to zero as $u \rightarrow \infty$, since it is the Laplace transform of a positive random variable. Consequently, the assumption $f(z)>0$ for arbitrarily large $z>0$ leads to the contradiction $\infty>\infty$. Rather, we must have that $f(z)=0$ for $z \geq T$ with some finite $T$, which turns the last inequality into
\begin{gather*}
\infty>\beta \,\int_0^{T}1\,\mathrm{d}s = \beta\,T,
\end{gather*}
and is not a contradiction.
\end{proof}

\begin{proof}[of Lemma \ref{lemma_killing}]
Sufficiency of the conditions (a) and (b): By Lemma~\ref{BF_mai18}, $\Psi_F$ has killing if and only if the left-end point of the support of $X \sim F$ is strictly positive. In this case, $\Psi_F(x)>\epsilon>0$ for all $x>0$ and one immediately observes from (\ref{anotherrepr_PsiH}) that $H$ has killing in this case (since $L$, hence $\nu_L$ is non-zero by assumption). Also, if $\Psi_L$ has killing, $\Psi_H(x)>\epsilon>0$ from (\ref{onerepr_PsiH}) for all $x>0$, hence $H$ has killing. 
\par
Necessity of the conditions (a) and (b): Assume that neither (a) nor (b) holds. Denoting the L\'evy measure associated with the Bernstein function $\Psi_F$ by $\nu_F$, (\ref{anotherrepr_PsiH}) can be re-written as
\begin{align*}
\Psi_H(x) &= \int_{(0,\infty]}\int_{(0,\infty]}\big( 1-e^{-x\,u\,s}\big)\,\nu_F(\mathrm{d}s)\,\nu_L(\mathrm{d}u)\\
&=\int_{(0,\infty)}\int_{(0,\infty)}\big( 1-e^{-x\,u\,s}\big)\,\nu_F(\mathrm{d}s)\,\nu_L(\mathrm{d}u),
\end{align*}
where the last equality follows from the assumptions that $\nu_L(\{\infty\})=\nu_F(\{\infty\})=0$. Taking the limit as $x \searrow 0$ on both sides of this equation, and using the bounded convergence theorem, it follows that $\lim_{x \searrow 0}\Psi_H(x)=0$, so $H$ has no killing.
\end{proof}

\begin{proof}[of Lemma \ref{lemma_filt}]
We denote $L_t:=\sum_{j=1}^{N_t}J_j$ with jump sizes $\{J_j\}_{j \geq 1}$ and Poisson process $N$, whose sequence of jump-times we denote by $\{\tau_k\}_{k\in\IN}$.
\par
The inclusion `$\subseteq$' is immediate from the representation $H_t=\sum\limits_{k:\tau_k\leq u_F\,t}-\log\big(F\big(\frac{\tau_k}{t}-\big)\big)J_k$.
\par
The inclusion `$\supseteq$' requires us to recover the jump times $\{\tau_k\}$ and jump sizes $\{J_k\}$ of the process $L$ up to time $u_F\,t$ from observing the process $H$ up to time $t>0$. This can be done inductively. 
\par 
If $H_t=0$, then $N_{u_F\,t}=0$ and we are done. Else, we recover the first jump time by $\tau_1:=\inf\{s\in[0,t]:H_s>0\}$ as well as $J_1=\lim_{\epsilon \searrow 0}H_{\tau_1+\epsilon}/-\log\big(F\big(\frac{\tau_1}{\tau_1+\epsilon}\big)\big)$.


\par 
If $H_t=-\log\big(F\big(\frac{\tau_1}{t}-\big)\big)\,J_1$, then $\tau_2>u_F\,t$ and we are done. Else, we recover the second jump time by $\tau_2:=\inf\{s\in(\tau_1,t]:H_s+\log\big(F\big(\frac{\tau_1}{s}-\big)\big)\,J_1>0\}$ as well as 
\begin{align*}
J_2=\lim_{\epsilon \searrow 0}\frac{H_{\tau_2+\epsilon}+\log\big(F\big(\frac{\tau_1}{\tau_2+\epsilon}\big)\big)\,J_1}{-\log\big(F\big(\frac{\tau_2}{\tau_2+\epsilon}\big)\big)}.
\end{align*}
This procedure is repeated until $\tau_n>u_F\,t$.
\end{proof}

\begin{proof}[that $z \mapsto H_z$ is holomorphic in Example \ref{ex_galambos}]
Using \cite[Theorem~7.2, p.~124]{gilman07}, it is sufficient to prove that the defining series of $H_z$ converges uniformly on all compact subsets of $\mathbb{C}_+$. Each compact subset of $\mathbb{C}_+$ is contained within a set of the form $\{z \in \mathbb{C}_+\,:\,\epsilon \leq |z| \leq c,\,\mathcal{R}(z)\geq \epsilon\}$ for some constants $\infty>c>\epsilon>0$. On this set, we have for arbitrary real $x>0$ that 
\begin{gather}
\Big| e^{-\frac{x}{z}}\Big| = e^{-\frac{x\,\mathcal{R}(z)}{|z|^2}} \leq e^{-\frac{x\,\epsilon}{c^2}}.
\label{shortref_galambos}
\end{gather}   
Furthermore, the law of the iterated logarithm implies that (a.s.)
\begin{align*}
\liminf_{k \rightarrow \infty}\frac{\tau_k-k}{\sqrt{2\,k\,\log(\log(k))}}=-1.
\end{align*} 
Consequently, we obtain for almost all $k \in \IN$ the estimate
\begin{align}
\tau_k = k+\sqrt{2\,k\,\log(\log(k))}\,\frac{\tau_k-k}{\sqrt{2\,k\,\log(\log(k))}} \geq  k-\sqrt{2\,k\,\log(\log(k))}\,(1+\epsilon) \geq k/2.
\label{shortref_galambos2}
\end{align}  
These estimates, together with the identity $-\log(1-x)=\sum_{m \geq 1}x^m/m$ for real $x \in [0,1)$, imply for $n$ large enough that
\begin{align*}
\Big|\sum_{k \geq n}-\log\Big( 1-e^{-\frac{\tau_k}{z}}\Big)\Big| & \leq \sum_{k \geq n} \sum_{m \geq 1}\frac{1}{m}\,\Big|e^{-\frac{\tau_k\,m}{z}}\Big| \stackrel{(\ref{shortref_galambos})}{\leq} \sum_{k \geq n} \sum_{m \geq 1}\frac{1}{m}\,e^{-\frac{\tau_k\,m\,\epsilon}{c^2}}\\
& \stackrel{(\ref{shortref_galambos2})}{\leq} \sum_{k \geq n} \sum_{m \geq 1}\frac{1}{m}\,e^{-\frac{k\,m\,\epsilon}{2\,c^2}} = \sum_{m \geq 1}\frac{1}{m}\,\frac{e^{-\frac{n\,m\,\epsilon}{2\,c^2}}}{1-e^{-\frac{m\,\epsilon}{2\,c^2}}}\\
& \leq \frac{1}{1-e^{-\frac{\epsilon}{2\,c^2}}}\,\frac{e^{-\frac{n\,\epsilon}{2\,c^2}}}{1-e^{-\frac{n\,\epsilon}{2\,c^2}}}.
\end{align*}
Since the last expression tends to zero as $n \rightarrow \infty$, uniform convergence of the defining series on all compact subsets of $\mathbb{C}_+$ is shown.
\end{proof}

\begin{proof}[of Lemma \ref{Gumbel_generalized}]
We approximate $F \in \F$ by the sequence of distribution functions $F_n(x):=\I{\{x\geq n\}}+\I{\{x< n\}}\,F(x)$, for $n \in \IN$, and observe that
\begin{gather*}
H^{(n)}_t:=H_t^{(F_n,L)}=\int_0^n-\log F (s/t -)\,\mathrm{d}L_s \longrightarrow H_t=H_t^{(F,L)} \mbox{ for }n\rightarrow\infty.
\end{gather*}
This, in turn, implies that $H_{t+x}^{(n)}-H_t^{(n)}\longrightarrow H_{t+x}-H_t$ for $(n\rightarrow\infty)$. As shown in the considerations preceding the lemma, the random variable $H_{t+x}^{(n)}-H^{(n)}_t$ equals the sum of a (non-negative) part $X^{(n)}_2$ that is measurable with respect to $\mathcal{F}^{H^{(n)}}_t$, and another (non-negative) part $X_1^{(n)}$ which is independent thereof and infinitely divisible. We denote the Bernstein function associated with $X_1^{(n)}$ by $\Psi_1^{(n)}$ and compute, using (\ref{bernstein_psi1}) and (\ref{LevyKhinchin}), that 
\begin{align*}
\Psi^{(n)}_1(\alpha)&= x\,\int_0^n \Psi_L\Big(-\alpha\,\log F\big(\frac{x\,y+n\,t}{x+t}\big)\Big) \,\mathrm{d} y, \\
	&= x\,\int_0^n \int_{(0,\infty]}\Big(1-F\big(\frac{x\,y+n\,t}{x+t}\big)^{\alpha z}\Big)\, \nu_L(\mathrm{d}z) \,\mathrm{d} y, \\
	&= (x+t)\,\int_{n\,t/(t+x)}^n \int_{(0,\infty]}\big(1-F(s)^{\alpha z}\big)\, \nu_L(\mathrm{d}z) \,\mathrm{d} y, \\
	&\leq (x+t)\,n\,\big(1-\frac{t}{t+x}\big) \int_{(0,\infty]}\Big(1-F\big(\frac{n\,t}{t+x}\big)^{\alpha z}\Big)\, \nu_L(\mathrm{d}z) \,\mathrm{d} y, \\
	&=x\, \frac{t+x}{t}\, \frac{t\,n}{t+x}\,\int_{(0,\infty]}\Big(1-F\big(\frac{n\,t}{t+x}\big)^{\alpha z}\Big)\, \nu_L(\mathrm{d}z) \,\mathrm{d} y \quad \stackrel{(\ref{usefulinequality3})}{\longrightarrow} 0, \quad (n\rightarrow\infty).
\end{align*}
Consequently, $X_1^{(n)}$ converges to zero in law, and hence converges to zero almost surely, since the limit is a constant. This implies that $X_2^{(n)}$ converges to $H_{t+x}-H_t$. Since $X_2^{(n)}$ is measurable with respect to $\mathcal{F}^{H^{(n)}}_t$ for all $n$, we conclude that $H_{t+x}-H_t$ is measurable with respect to $\mathcal{F}^H_t$,  yielding the claim.
\end{proof}

\begin{proof}[of Lemma \ref{lemma_simpleadmissible}]
The conditions on $\Psi_F$ imply that it is the Bernstein function associated with a compound Poisson subordinator whose jump size distribution has finite mean. Example~\ref{ex_CPP} shows that $\F_{L^{(F)}}=\F$. In particular, $(F^{(L)},L^{(F)})$ is admissible, which implies admissibility of $(F,L)$ by the aforementioned duality.
\end{proof}

\begin{proof}[of Lemma \ref{lemma_bondesson}]
\begin{itemize}
\item[(a)] We compute, using (\ref{Psi_H_B}), $\Psi_L(x)=x/(x+1)$, and (\ref{stieltjesbiject}) that
\begin{align*}
\Psi_H(x) &=  \int_0^{\infty}\Psi_L(x\,u)\,\nu_F(\mathrm{d}u)= \int_0^{\infty}\frac{x\,u}{x\,u+1}\,\nu_F(\mathrm{d}u)  =   \int_0^{\infty}\frac{x}{x+\frac{1}{u}}\,\nu_F(\mathrm{d}u) \\
& =  \int_0^{\infty}\frac{x}{x+u}\,\rho_F(\mathrm{d}u).
\end{align*}
\item[(b)] We recall from Lemma~\ref{BF_mai18} that $F \mapsto \nu_F$ is a bijection between $\hat{\F}$ and L\'evy measures on $(0,\infty)$. Furthermore, it is clear that the mapping $\nu_F \mapsto \rho_F$ is injective. Left to show is only that every Stieltjes measure is attainable, i.e.\ that $\nu_F \mapsto \rho_F$ is surjective. To this end, it is sufficient to show for an arbitrary Stieltjes measure $\rho$ that the measure 
\begin{gather*}
\nu(A):=\rho\Big(\Big\{\frac{1}{u}\,:\,u \in A \Big\} \Big),\quad A \in \mathcal{B}\big( (0,\infty)\big),
\end{gather*}
is a L\'evy measure on $(0,\infty)$. To this end, we observe that
\begin{align*}
\nu\big( (1,\infty)\big)&=\rho\big( (0,1)\big) \leq \int_0^{1}\frac{2}{1+t}\,\rho(\mathrm{d}t) \leq 2\,\int_0^{\infty}\frac{1}{1+t}\,\rho(\mathrm{d}t) <\infty,\\
\int_0^{1}u\,\nu(\mathrm{d}u) & = \int_1^{\infty}\frac{1}{t}\,\rho(\mathrm{d}t) \leq \int_1^{\infty}\frac{2}{t+1}\,\rho(\mathrm{d}t)  \leq 2\,\int_0^{\infty}\frac{1}{t+1}\,\rho(\mathrm{d}t)<\infty.
\end{align*}
The claimed expression for $F_{\rho}$ can be retrieved directly from Lemma~\ref{BF_mai18}, while noticing that $\rho_F\big((0,1/x)\big)=\nu_F\big( (x,\infty)\big)$.
\item[(c)] This is a direct consequence of part (b) and the definition of the process $H$.
\end{itemize}
\end{proof}

\end{document}